\documentclass[leqno]{amsart}
\usepackage{amsfonts,amssymb,amsmath,amsgen,amsthm}
\usepackage{hyperref}
\usepackage{color}
\newcommand{\msc}[2][2000]{%
  \let\@oldtitle\@title%
  \gdef\@title{\@oldtitle\footnotetext{#1 \emph{Mathematics subject
        classification.} #2}}% 
}

\theoremstyle{plain}
\newtheorem{theorem}{Theorem} [section]
\newtheorem{definition}[theorem]{Definition}
\newtheorem{assumption}[theorem]{Assumption}
\newtheorem{lemma}[theorem]{Lemma}
\newtheorem{corollary}[theorem]{Corollary}
\newtheorem{proposition}[theorem]{Proposition}

\theoremstyle{remark}
\newtheorem{remark}[theorem]{Remark}

\newtheorem{example}[theorem]{Example}

\def\R{{\mathbb R}}% real numbers
\def\N{{\mathbb N}}% nonnegative integers
\def\Z{{\mathbb Z}}% integers
\def\Sch{{\mathcal S}}% Schwartz space
\def\O{\mathcal O}

\def\Op{\operatorname{Op}^w}
\def\cH{\mathcal{H}}
\def\cT{\mathcal{T}}
\def\cZ{\mathcal{Z}}
\def\uC{\underline C}
\def\uK{\underline K}

\def\vp{W}
\def\pp{{\mathbf p}}
\def\qq{{\mathbf q}}

\def\({\left(}
\def\){\right)}
\def\<{\left\langle}
\def\>{\right\rangle}
\def\le{\leqslant}
\def\ge{\geqslant}

\def\Tend#1#2{\mathop{\longrightarrow}\limits_{#1\rightarrow#2}}

\def\eps{\varepsilon}

\def\si{{\sigma}}

\def\Id{\mathbf 1}

\numberwithin{equation}{section}

\begin{document}
\title[Time splitting for NLS with potential]{Time
  splitting and error estimates for nonlinear  Schr\"odinger equations
  with a potential}

\author[R. Carles]{R\'emi Carles}
\address{CNRS\\ IRMAR - UMR 6625\\ F-35000
  Rennes, France}
\email{Remi.Carles@math.cnrs.fr}

\begin{abstract}
 We consider the nonlinear Schr\"odinger equation with a potential,
 also known as  Gross-Pitaevskii equation. By introducing a suitable
 spectral localization, we prove low regularity error estimates for
 the time 
 discretization corresponding to an adapted Lie-Trotter splitting
 scheme. The
 proof is based on tools from spectral theory and pseudodifferential
 calculus in order to obtain various estimates on the spectral
 localization, including discrete Strichartz estimates
which support the nonlinear analysis. 
\end{abstract}

\thanks{This work was supported by Centre Henri Lebesgue,
program ANR-11-LABX-0020-0. A CC-BY public copyright license has
  been applied by the author to the present document and will be
  applied to all subsequent versions up to the Author Accepted
  Manuscript arising from this submission.}

\maketitle

\section{Introduction}
\label{sec:intro}

\subsection{Setting}
\label{sec:setting}
We consider time discretization for the nonlinear Schr\"odinger equation with a potential,
\begin{equation}\label{eq:NLSP}
  i\partial_t u +\Delta u= V(x)u +\vp (x)u+ \eps |u|^{2\si}u\quad ;\quad
  u_{\mid t=0}=u_0,
\end{equation}
with $x\in \R^d$, $d\ge 1$, $\eps\in \{-1,+1\}$, and $\si>0$. More
precisely, the nonlinearity is $H^1$-subcritical, $0<\si<\frac{2}{(d-2)_+}$, that is,
  $\si<\frac{2}{d-2}$ if $d\ge 3$. We denote
\begin{equation*}
  H = -\Delta +V(x). 
\end{equation*}
A typical physically relevant case in dimension $d\le 3$, related to
the physics of superfluids or Bose-Einstein Condensation, is the cubic
nonlinearity ($\si=1$), with a quadratic confining potential (possibly
anisotropic),
 \begin{equation}\label{eq:pot-harmo}
      V_{\rm quad}(x) = \sum_{j=1}^d\omega_j^2x_j^2,\quad \omega_j>0.
    \end{equation}
Equation~\eqref{eq:NLSP} in that case may be referred to as
Gross-Pitaevskii equation; see e.g. \cite{BaoCai2013,JosserandPomeau} and
references therein.  The potential $\vp$ 
may be considered as a perturbation of $V$. A major difference between
the two potentials, see Assumption~\ref{hyp:V}, is that $V$ is smooth,
possibly unbounded,
while $\vp$ is in $W^{2,\infty}(\R^d)$. We emphasize that we make no
assumption on the spectrum of the Hamiltonian $H$: for instance, $V$
may correspond to a partial confinement, that is we may consider
\eqref{eq:pot-harmo} in the case were some (or all) $\omega_j$'s are zero. 
  The potential in \eqref{eq:pot-harmo}  is not bounded, and cannot be
  addressed by 
  perturbative arguments, typically when solving the Cauchy problem
  \eqref{eq:NLSP} and considering the associated dynamics. In this
  paper, and in agreement with the above example,
  we assume that $V$ is bounded 
  from below,
   \begin{equation*}
    \exists C_0>0,\quad V(x)\ge -C_0,\ \forall x\in \R^d.
  \end{equation*}
  The gauge transform $u\mapsto e^{-it(C_0+1)} u=\tilde u$ shows that
  $\tilde u$ solves \eqref{eq:NLSP} with $V$ replaced by $V+C_0+1$, so
  we may  assume that $V\ge 1$ without loss of generality. This will simplify some
  technical arguments, without discarding potentials such as in
  \eqref{eq:pot-harmo}:
\begin{assumption}\label{hyp:V}
  The potential $V$ is real-valued, smooth, $V\in
  C^\infty(\R^d;\R)$, with $V\ge 1$. Moreover, it is at
  most quadratic, in the sense that  
  \begin{equation*}
   \forall \alpha \in \N^d, \ |\alpha|\ge 2, \quad
    \partial^\alpha V\in L^\infty( \R^d).
  \end{equation*}
 The potential $\vp$ is real-valued, $\vp\in W^{2,\infty}(\R^d;\R)$, that is
 \begin{align*}
  \forall \alpha \in \N^d, \ |\alpha|\le 2, \quad \partial^\alpha \vp\in L^\infty(\R^d).
 \end{align*}
\end{assumption}
The operator $H$ is self-adjoint, and we denote by
\begin{equation}\label{eq:S(t)}
  S(t)=e^{-itH}
\end{equation}
the evolution group associated with the linear part in
\eqref{eq:NLSP} ($\eps=0$), in the case $\vp=0$, which is well defined under
Assumption~\ref{hyp:V} 
(see e.g. \cite{ReedSimon2}).

\subsection{Time discretization}

This paper is focused on error estimates for the Lie-Trotter time
splitting scheme associated to \eqref{eq:NLSP}, see
e.g. \cite{AntoineBaoBesse2013,BaoCai2013} and references therein. As
in \eqref{eq:NLSP}, 
$\partial_t u$ is the sum of four terms, several strategies appear
naturally if one wants to write
\begin{equation*}
  \partial_t u = A_1u +A_2u,
\end{equation*}
where the equation $\partial_t u= A_ju$ is convenient to solve numerically
both for $j=1$ and $j=2$. We do not consider more sophisticated
splitting methods here, to avoid extra technicalities. One may set $A_1=i\Delta$, and
\begin{equation*}
  A_2 u = -iV u -i\vp u -i\eps|u|^{2\si}u.
\end{equation*}
This choice was made in, e.g.,
\cite{BaoJakschMarkowich2003}, in the case where $V$ is
quadratic like in \eqref{eq:pot-harmo}: $e^{tA_1}$ is computed by
Fourier pseudospectral methods, and the nonlinear ordinary
differential equation
\begin{equation*}
  \partial_t u = A_2 u
\end{equation*}
is solved explicitly, as it turns out to be a linear equation, since
$\partial_t |u|^2=0$ ($\eps$ is real). Error estimates are proved in
\cite{Thalhammer2012}, assuming large regularity for the initial data,
but also considering a bounded spatial domain (hence the potential is
bounded). 
\smallbreak

In the case $\vp=0$, the other natural choice consists
in setting $A_1=-iH$ and $A_2 u = -i\eps|u|^{2\si}u$. This choice was
made in, e.g., \cite{BaoShen2005}, again for $V$ of the form
\eqref{eq:pot-harmo}, and using the Hermite eigenbasis to replace
Fourier pseudospectral methods with Hermite pseudospectral methods:
the operator $e^{tA_1}$ is computed on each Hermite mode, see
Example~\ref{ex:eigendecomp-harmo} below (with $\chi\equiv 1$). In
\cite{ThCaNe2009}, the two decompositions were addressed 
numerically (for a higher order splitting method), and no crucial
difference seemed to emerge regarding the size of the error. In the isotropic
case of \eqref{eq:pot-harmo} ($\omega_j=1$ for all $j$), an error
analysis in the case of Strang splitting scheme was provided in
\cite{Gaukler2011} for the cubic nonlinearity $\si=1$,
in the spirit of the proof in \cite{Lu08} for the case without
potential, requiring a high level of 
regularity for the initial data, and working in Sobolev spaces based
on powers of $H$ instead of powers of $-\Delta$, that is, accounting
for the presence of the potential. 
\smallbreak

It is possible to mix the above two approaches, as in
e.g. \cite{ThCaNe2009}, in the following sense. Suppose for instance
that $V(x) = V_{\rm quad}(x) + e^{-|x|^2} +1$. Then it is possible to
consider $A_1 =i\Delta-iV_{\rm quad}-i$ and $A_2u =  -i e^{-|x|^2}u
-i\eps|u|^{2\si}u$. The flow generated by $A_1$ can be analyzed by
spectral theory like in e.g. \cite{Gaukler2011}, while the potential in
$A_2$ is in the Schwartz class, and can analyzed by perturbative
arguments. 
\smallbreak

The framework that we consider here is in the spirit of the latter
option, and we set $A_1=-iH$, $A_2 = -i \vp u
-i\eps|u|^{2\si}u$.  This distinction can be of practical interest,
even for $\vp$ in the Schwartz class, as 
  the spectral cutoff $\Pi_\lambda$ defined in
\eqref{eq:Pi_lambda} is explicit for particular potentials only. 
We
emphasize that even in the case where $V$ is quadratic, the tools we
present in Section~\ref{sec:tools}  may be interesting in other
contexts. Our main contribution, compared to the
results evoked above, consists in
decreasing the regularity required on the initial data, thanks to
Strichartz estimates (but we stick to
the Lie-Trotter discretization and do not address Strang
splitting). The motivation is theoretical and practical: the nonlinearity need not be
smooth ($\si$ can be arbitrarily small), and we know, in some cases
(typically when $\si<2/d$ -- $L^2$-subcritical nonlinearity -- or
$\eps>0$ -- defocusing nonlinearity), that the solution to
\eqref{eq:NLSP} is defined globally in time (at the level of
regularity that we need), so the main error estimates are true on
arbitrary time intervals (see Remark~\ref{rem:T}). 
\smallbreak

We define $N(t)
\phi$ as the solution of the flow 
    \begin{equation*}
       i\partial_t u  =\vp u+\eps |u|^{2\si} u,\quad
    u_{\mid t=0}= \phi,
    \end{equation*}
that is, $N(t) \phi = \phi e^{-i t \vp -it\eps |\phi|^{2\si}}$. As
Equation~\eqref{eq:NLSP} is time
reversible, we shall only consider positive time, the case of negative
time being similar. The standard
Lie-Trotter approximation is 
defined, for $\tau\in (0,1)$,  as
\[Z(n \tau) \phi = \( S(\tau) N(\tau) \)^n \phi,\quad n\in \N,\]
where $S$ is defined in \eqref{eq:S(t)}. 
Note that in the case where $V$ is harmonic, \eqref{eq:pot-harmo}, $S(\cdot)$
can be computed by using Hermite functions, see e.g. \cite{BaoCai2013},
and Example~\ref{ex:eigendecomp-harmo} below. 
\begin{remark}
  The potential $\vp$ could include a singularity, in the sense that
  the assumption $\vp\in L^\infty(\R^d)$ could be replaced by $\vp\in
  L^p(\R^d)+L^\infty(\R^d)$, for some $p\ge 1$ such that $p>d/2$,
  typically (for suitable $\gamma$ and $p$)
  \begin{equation*}
    \vp(x) = \underbrace{\frac{1}{|x|^{\gamma} }\Id_{|x|<1}}_{\in
      L^p(\R^d)}+  \underbrace{ \frac{1}{|x|^{\gamma}
      }\Id_{|x|\ge1}}_{\in L^\infty(\R^d)},
  \end{equation*}
like in
  \cite{CazCourant}, with similar assumptions on $\partial^\alpha \vp$ for
  $|\alpha|\le 2$. However, if $\vp$ is singular,  the nonlinear flow
  $N(t)$ becomes more delicate to handle numerically (when space
  discretization is considered too), this is why we
  simply assume $\vp\in W^{2,\infty}(\R^d)$.
\end{remark}

\smallbreak

 In the case $V=\vp=0$, error
estimates related to $Z$ were first proven in \cite{BBD}, for globally
Lipschitz nonlinearities. The proof
was made more systematic, for other nonlinearities,  in the more
difficult case of Strang 
splitting, in \cite{Lu08}. In both papers, the regularity requested on
the initial data, in order to obtain an error estimate in $L^2$, is
rather high ($H^2$ in \cite{BBD} for Lie, $H^4$ in \cite{Lu08} for
Strang -- in both papers, the space dimension is sufficiently low so
these Sobolev spaces are Banach algebras embedded into
$L^\infty$). This regularity constraint was relaxed in \cite{Ignat2011} and
\cite{ChoiKoh2021}, by using discrete in time Strichartz estimates. We
give more details regarding this tool in Section~\ref{sec:strichartz},
which is crucial in our 
case too, where $V$ need not be zero. As pointed out  in
\cite{Ignat2011,IgnatZuazua2006,IgnatZuazua2009,StKe05}, in the case
$V\equiv 0$, the group  $S(\cdot)$ does not satisfy
discrete in time Strichartz estimates: a frequency cutoff is needed,
see also \cite{ORS21}. When $V$ is present, this frequency cutoff must
be replaced by a more general operator.

\smallbreak

To prove discrete in time Strichartz estimates, we need to remove the
singularity at $t=0$ of the dispersive rate in $e^{-itH}$. In the case without
potential, this is achieved by a frequency cutoff, see
e.g. \cite{Ignat2011,ORS21}. In the present framework,
the frequency cutoff is replaced  by a spectral cutoff, in order
typically to keep the commutation property between the cutoff $\Pi$ and the
exact flow  $S(t)$: for potentials like in Assumption~\ref{hyp:V}, the
lack of commutation between the Fourier multiplier considered in
\cite{Ignat2011} and $H$ may generate error terms which cannot be
controlled (typically in the case \eqref{eq:pot-harmo}). 
\smallbreak

Throughout this paper, $\chi\in C^\infty_c(\R;[0,1])$ denotes a smooth,
compactly supported
function which is one on $[-1,1]$, and zero outside $[-2,2]$. For
$\lambda>0$, we set 
\begin{equation}\label{eq:Pi_lambda}
  \Pi_\lambda = \chi^2\(\frac{H}{\lambda}\),
\end{equation}
the spectral localization defined by functional calculus, see
Section~\ref{sec:tools}.  In particular,
Proposition~\ref{prop:funct-calc} implies that $\Pi_\lambda$ and $H$
commute, therefore so do $\Pi_\lambda$ and $S(t)$.

We emphasize the fact that when $V=\vp=0$ (like in
\cite{Ignat2011,ChoiKoh2021,ORS21}), the Fourier multiplier
$\Pi_\lambda$ can be viewed as a convolution operator, which yields
various estimates thanks to Young inequality. When $V\not =0$,
$\Pi_\lambda$ is no longer  a 
Fourier multiplier: we shall use the fact that since $\chi\in
C^\infty_c$, $\Pi_\lambda$ is a pseudodifferential operator, thanks to
a result due to Helffer and Robert \cite{HelfferRobert1983}, and rely
on properties of pseudodifferential operators (see
Section~\ref{sec:tools}). 
The generalization of the truncated free group from \cite{Ignat2011,ORS21}
is then
  \begin{equation}\label{eq:S_lambda}
    S_\lambda(t) = \Pi_\lambda S(t) = S(t)\Pi_\lambda,
  \end{equation}
and the modified Lie-Trotter splitting scheme reads
   \begin{equation}\label{eq:modified-splitting}
   u^n =  {\mathcal Z}_{\lambda} (n\tau)u_0 := \( S_{\lambda} (\tau) N(\tau) \)^n \Pi_{\lambda}u_0.
    \end{equation}
At this stage, we have not related $\tau$ and $\lambda$. We will
eventually consider $\lambda=1/\tau$, which is the same choice as
in \cite{Ignat2011} and \cite{ChoiKoh2021}. Roughly speaking, this
makes it possible to show $L^2$ error estimates of order $\tau^{1/2}$ on
bounded time intervals, for $H^1$ solutions, and of order $\tau$ for
$H^2$ solutions. We point out however that
the setting we introduce makes it possible to adapt the strategy
developed in \cite{ORS21}, where other relations between $\lambda$
(denoted by $K^2$ there) and $\tau$ are proposed, for another
integrator, in order to break the natural order barrier of
$\tau^{1/2}$ error estimates for $H^1$ solutions. 
%  \begin{example}[Harmonic potential] When $V=V_{\rm quad}+1$,
% we have a well-known eigenbasis consisting of tensor products of
% Hermite functions   associated to the frequency $\omega_j$, which
% make it possible to compute $S(t)$  and $S_\lambda(t)$ rather directly.
%  \end{example}

 % \begin{remark}[Bounded potential]
 %  In the case where $V$ is bounded as well as its derivatives,
 %   \begin{equation*}
 %     \chi\(\frac{H}{\lambda}\)-
 %     \chi\(\frac{-\Delta}{\lambda}\)=\O\(\frac{1}{\lambda}\),
 %   \end{equation*}
 % as an operator from $L^2(\R^d)$ to $L^2(\R^d)$, and we see that up
  %to an error which does not alter the statement of our main results
  %(Theorems~\ref{theo:main} and \ref{theo:main2} below), we can resume
  %the spectral cutoff from e.g. \cite{ChoiKoh2021,Ignat2011,ORS21}. 
  %\end{remark}
\begin{example}\label{ex:eigendecomp-harmo}
In the case of a harmonic potential as in \eqref{eq:pot-harmo}, a
well-known eigenbasis is given by Hermite functions, and the operator
$\Pi_\lambda$ can be written in terms of the
eigendecomposition. If $d=\omega_1=1$, with $V=V_{\rm
  quad}+1=x^2+1$ (the constant $1$ is added to be
consistent with Assumption~\ref{hyp:V}), Hermite functions 
$\varphi_j$, $j\in 
  \N$, solve $H\varphi_j = 
  \(3+2j\)\varphi_j$, and, if they
  are normalized so that $\|\varphi\|_{L^2}=1$, every $L^2$ function
  $\phi$ can be decomposed as
  \begin{equation*}
    \phi(x) = \sum_{j\ge 0}\<\phi,\varphi_j\>\varphi_j(x),\quad
    \<\phi,\varphi_j\>=\int_{\R}\phi(y)\bar\varphi_j(y)dy.
  \end{equation*}
Then we may also write
\begin{equation*}
  \Pi_\lambda \phi(x) = \sum_{j\ge 0}\chi\(
  \frac{\lambda_j}{\lambda}\)^2\<\phi,\varphi_j\>\varphi_j(x),
  \quad \lambda_j=3+2j,
\end{equation*}
and
\begin{equation*}
  S_\lambda(t) \phi(x) = \sum_{j\ge 0}e^{-it\lambda_j}
\chi\( \frac{\lambda_j}{\lambda}\)^2\<\phi,\varphi_j\>\varphi_j(x).
\end{equation*}
The expression of $S(t)$ is obtained by setting $\chi\equiv 1$.
The presence of the spectral cutoff $\chi$ can be understood as a
discretization in space, since we consider only finitely many Hermite
functions, like in \cite{Gaukler2011}. We will see below that $\lambda$ is
related to the time 
step $\tau$, this relation may be understood as a CFL condition.
\end{example}

\subsection{Main results}
\label{sec:main}

The first $L^2$-based function space we consider is 
\begin{equation*}
  \cH^1 =\{\phi\in H^1(\R^d);\ \phi\sqrt V \in L^2(\R^d)\},
\end{equation*}
equipped with the norm
\begin{align*}
    \|\phi\|_{\cH^1}^2& :=\<H\phi ,\phi\> =
     \|\nabla \phi\|_{L^2(\R^d)}^2+ \int_{\R^d}V(x)|\phi(x)|^2dx.
\end{align*}
Note that in view of Assumption~\ref{hyp:V}, since $V\ge 1$, 
$\|\phi\|_{\cH^1}^2\ge \|\phi\|_{L^2}^2$. Also, $\Sigma\subset \cH^1$, where
  \begin{equation*}
  \Sigma:=\left\{\phi\in H^1(\R^d);\
    \int_{\R^d}|x|^2|\phi(x)|^2dx<\infty\right\}. 
\end{equation*}
In the case of a quadratic potential \eqref{eq:pot-harmo}, the two
spaces coincide, and correspond to the sharp analogue of the Sobolev space
$H^1(\R^d)$ compared to the case 
$V=0$, see e.g. \cite{Ca15}. Under Assumption~\ref{hyp:V},
\eqref{eq:NLSP} is locally well-posed in $\Sigma$, see e.g. \cite{Ca11}. In 
Section~\ref{sec:cauchy}, we prove well-posedness results in the
possibly larger space $\cH^1$. 

Like in \cite{ChoiKoh2021}, the statements of our results contain a
restriction on the dimension, 
since when $\si\ge 2/d$, we
assume in addition $\si\ge 1/2$ for the nonlinearity in
\eqref{eq:NLSP} to be of class $C^2$ (or $\dot W^{2,\infty}$ in the
case $\si=1/2$), which implies $d\le 5$ since $\si<2/(d-2)_+$. This extra
condition is therefore only present when $d=5$. Note that our
results cover the physical case of a cubic nonlinearity in 
dimension $d\le 3$.

\begin{theorem}\label{theo:main}
  Let Assumption~\ref{hyp:V}  be verified.
Assume that either $0<\si<2/d$, or $d\le 5$ and $2/d\le
\si<2/(d-2)_+$, with in addition $\si\ge 1/2$ when $d=5$. Let  $u_0\in
\cH^1$, assume that \eqref{eq:NLSP} has a unique 
solution $u\in C([0,T];\cH^1)$, for some $T>0$, and denote by $u^n$ the sequence
defined by the scheme \eqref{eq:modified-splitting}, where we set
$\lambda=1/\tau$. There exist
$\tau_0>0$ and $C=C(T)>0$ such that for every $\tau\in (0,\tau_0]$, we
have the error estimate
\begin{equation*}
  \|u^n-u(n\tau)\|_{L^2(\R^d)} \le C\tau^{1/2},\quad 0\le n\tau\le T.
\end{equation*}
\end{theorem}

To prove a better error estimate in terms of the time step, we assume
higher regularity: let
  \begin{equation*}
  \cH^2:=\left\{\phi\in H^2(\R^d);\ V \phi \in L^2(\R^d) \right\}, 
\end{equation*}
equipped with the norm
\begin{equation*}
    \|\phi\|_{\cH^2}^2:=  \|\Delta\phi\|_{L^2(\R^d)}^2+ \| V
    \phi\|_{L^2(\R^d)}^2.
\end{equation*}

\begin{theorem}\label{theo:main2}
  Let Assumption~\ref{hyp:V}  be verified. Assume that  $1/2\le 
\si<2/(d-2)_+$. 
Let  $u_0\in \cH^2$, assume that \eqref{eq:NLSP} has a unique
solution $u\in C([0,T];\cH^2)$, for some $T>0$, and denote by $u^n$ the sequence
defined by the scheme \eqref{eq:modified-splitting}, where we set
$\lambda=1/\tau$. There exist
$\tau_0>0$ and $C=C(T)>0$ such that for every $\tau\in (0,\tau_0]$, we
have the error estimate
\begin{align*}
  &\|u^n-u(n\tau)\|_{L^2(\R^d)} \le C\tau,\quad 0\le n\tau\le T,\\
  &\|u^n-u(n\tau)\|_{\cH^1} \le C\tau^{1/2},\quad 0\le n\tau\le T.
\end{align*}
\end{theorem}

\begin{remark}[On the time $T$]\label{rem:T}
  As we will see in Section~\ref{sec:cauchy}, the time $T$ involved in
  Theorems~\ref{theo:main} and \ref{theo:main2} can be taken
  arbitrarily large (global solution) in two cases: $0<\si<2/d$
  ($L^2$-subcritical case) or $\eps=+1$ (defocusing case). When
  $\si\ge 2/d$ and $\eps=-1$ (focusing case), finite time blow up may
  occur (see e.g. \cite{Ca11,CazCourant}), so having a
  solution $u$  well-defined up to time $T$ becomes a nontrivial
  assumption, even though we know that a unique local in time solution
  always exists (for \emph{some} time $T>0$). In the case $\si\ge 2/d$
  and $\eps=-1$, we can prove that the solution is global provided
  that $\|u_0\|_{\cH^1}$ is sufficiently small. However, this
  theoretical statement does not yield an explicit smallness
  condition (unless $\si=2/d$, see e.g. \cite{CazCourant}), so it is
  not very convenient in practice. 
\end{remark}

\begin{remark}[Growth of the constants in the error estimates]
 Suppose that the solution to \eqref{eq:NLSP} is global in time (which
 is always granted if $\si<2/d$ or $\eps=+1$). In the case without
 potential, $V=0$, the dependence of the 
  constant $C(T)$ upon $T$ in (the analogue of) Theorems~\ref{theo:main} and
  \ref{theo:main2} has been tracked very carefully in
  \cite{ChoiKoh2021}. The proof yields a possible exponential
  growth, and our argument does not yield a slower growth. In some
  specific situation (extra properties related to 
  scattering theory when $V=0$), $C(T)$ can be taken uniform in $T$,
  \cite{CaSu24}. It is unlikely that a similar improvement can be
  expected in the case where $V$ is, e.g., a harmonic potential, since
  no strong dispersion is expected in this confining case.  
\end{remark}

\begin{remark}
  As pointed out in Example~\ref{ex:eigendecomp-harmo} in the case of
  the harmonic potential, the spectral cutoff $\chi$ may be understood
  as a spatial discretization, and the relation $\lambda=1/\tau$
  corresponds to a CFL condition. 
\end{remark}

\begin{remark}[Optimality of the estimates]
  In the case where $V$ is a harmonic potential, the numerical
  simulations from \cite{CaMa-p} show that the $\O(\tau)$ error
  estimate in Theorem~\ref{theo:main2} is sharp. Examining the
  sharpness of the estimate in Theorem~\ref{theo:main} and the second
  inequality in Theorem~\ref{theo:main2} would require careful
  simulations. 
\end{remark}

\begin{remark}[Nonautonomous equation]
In the case of a nonautonomous nonlinearity of the form
\begin{equation*}
  i\partial_t u +\Delta u= V(x)u + \vp(x)u+h(t) |u|^{2\si}u,
\end{equation*}
where the function $h$ is bounded on $[0,T]$, Theorems~\ref{theo:main}
and \ref{theo:main2} remain valid, as it is easy to check that the
introduction of $h$ does not change the error analysis. Moreover,
 the local Cauchy problem can be handled like in
Section~\ref{sec:cauchy}; global existence is not straightforward
though, as the conservation of energy is lost when $h$ is not
constant (see e.g. \cite[Section~4.11]{CazCourant}). 
\end{remark}

\subsection{Outline of the proof}
 Handling low regularity solutions $u$  follows the strategy introduced in
 \cite{Ignat2011} for $V=0$, and refined in \cite{ChoiKoh2021,ORS21}, based on the
 use of Strichartz estimates. Continuous in time Strichartz estimates
 are now classical (see e.g. \cite{CazCourant,KeelTao}), even in the
 case with a potential satisfying Assumption~\ref{hyp:V}, as recalled
 in Section~\ref{sec:strichartz}. In the case $V=0$, the discrete in
 time version requires a frequency cutoff, as shown in
 \cite{Ignat2011} and examined more thoroughly in \cite{ORS21}. For
 $V\not =0$, the notion to consider in order to generalize the
 frequency cutoff is at the heart of this paper. This cutoff
 $\Pi_\lambda$ must satisfy the commutation property
 \eqref{eq:S_lambda}, and various estimates involving $\Pi_\lambda$
 are needed in the course of the analysis.
 \smallbreak

 In Section~\ref{sec:tools},
 we introduce the technical tools whose use in this context
 (discretization in time) appears to be new. At the heart of the
 presentation lies the introduction of the spectral cutoff
 $\Pi_\lambda$, defined by functional calculus in
 Section~\ref{sec:funct}. This object turns out 
 to be a pseudodifferential operator, from \cite{HelfferRobert1983}:
 we list the properties related to Weyl-H\"ormander calculus which are
 used here, in Section~\ref{sec:pseudo}. In particular,
 Lemma~\ref{lem:sobolev-H} may have applications in other contexts,
 as it shows how to account for the presence of a potential satisfying
 Assumption~\ref{hyp:V} when studying e.g. Schr\"odinger
 equations. The main analytical properties of $\Pi_\lambda$ are
 established in Section~\ref{sec:PiPDO}.
 \smallbreak

 Strichartz estimates are
 stated in Section~\ref{sec:strichartz}. They are new only in the case
 of discretized time, which is described in more details, and relies
 on the properties of the spectral cutoff $\Pi_\lambda$. Analytical 
 properties of the exact (continuous) solution $u$ are given in
 Section~\ref{sec:cauchy}: 
 Lemma~\ref{lem:sobolev-H} is invoked several times, in order to adapt
 some results already 
 available in the case $V=0$, which can be found in
 e.g. \cite{CazCourant}.
 \smallbreak

 Sections~\ref{sec:tools}, \ref{sec:strichartz} and \ref{sec:cauchy}
 can be understood as central preparatory steps to adapt the tools from
 \cite{Ignat2011,ChoiKoh2021,ORS21} to the presence of the potential
 $V$. Their content is likely to be useful for other questions than
 the proof of Theorems~\ref{theo:main} and \ref{theo:main2}, and
 constitutes the main novelty of this article. In
 Section~\ref{sec:cv}, we show how stability implies 
 convergence: the scheme of the proof is the same as in
 \cite{ChoiKoh2021}, but the presence of the potential requires more
 involved estimates. Finally, Section~\ref{sec:stability} contains the
 proof of stability, by adapting the strategy of \cite{Ignat2011}
 based mostly on a bootstrap argument, whose details are a bit
 different here.

\subsection{Notations}
\label{sec:notations}

For $Y$ a Banach space, we shall denote
$\|f\|_{L^q_T Y}=\|f\|_{L^q(0,T;Y)}$.
The Japanese bracket is classically defined as 
$\<z\>=\(1+|z|^2\)^{1/2}$

\section{Technical tools} 
\label{sec:tools}

In this section, we gather technical results which are useful in the
rest of the paper, related to functional and pseudodifferential
calculus.  Most of the results that we shall use in the proof of
Theorems~\ref{theo:main} and \ref{theo:main2} are presented in
Subsection~\ref{sec:spectral}, and concern properties of the spectral
localization $\Pi_\lambda$. We will also rely on the equivalence of
norms stated in Lemma~\ref{lem:sobolev-H}. The perturbative potential
$\vp$ does not appear in this section.

\subsection{Preliminaries}
\label{sec:prelim}

We first state a 
consequence of Assumption~\ref{hyp:V}, which is 
obvious in the case of a quadratic potential:
\begin{lemma}\label{lem:4.1ZAMP}
  If $V$ satisfies
  Assumption~\ref{hyp:V}, then 
  \begin{equation*}
    |\nabla V(x)|^2\le  2d\|\nabla^2 V\|_{L^\infty}V(x),\quad \forall 
    x\in \R^d.
  \end{equation*}
\end{lemma}
\begin{proof}
  The property $V\ge 1$ is not needed for this lemma, as $V\ge 0$ suffices.
  Taylor's formula yields, for $x,y\in \R^d$,
  \begin{align*}
    V(x+y) &= V(x) +y\cdot \nabla V(x) + \int_0^1
             (1-\theta)\<y,\nabla^2 V(x+\theta y)y\>d\theta\\
    &\le V(x) + y\cdot
    \nabla V(x) + \frac{|y|^2}{2}\|\nabla^2 V\|_{L^\infty}.  
  \end{align*}
  Since $V(x+y)\ge 0$, the polynomial in $y$ on the right hand side is
  everywhere nonnegative. Taking for instance $y=y_je_j$, where
  $(e_k)_{1\le k\le d}$ is the canonical basis of $\R^d$, the
  discriminant of $V(x) +y_j \partial_j V(x) + \frac{y_j^2}{2}\|\nabla^2
  V\|_{L^\infty}$ is nonpositive, hence
 \begin{equation*}
    |\partial_j V(x)|^2\le  2\|\nabla^2 V\|_{L^\infty}V(x),\quad \forall
    x\in \R^d.
  \end{equation*}
  The result follows by summing over $j$.
\end{proof}

We will also invoke the following result, known either as Schur test,
or Young inequality for integral operators, as stated in
\cite[Theorem~0.3.1]{Sogge2017}: 
\begin{proposition}\label{prop:schur}
  For $d\ge 1$, consider an operator $\cT$ with Schwartz kernel $K$,
  \begin{equation*}
    \cT\phi(x) = \int_{\R^d}K(x,y)\phi(y)dy.
  \end{equation*}
Let $1\le p\le q\le \infty$ and $r\ge 1$ be related by
\begin{equation*}
  \frac{1}{r} = 1-\frac{1}{p}+\frac{1}{q}.
\end{equation*}
If $K$ satisfies
\begin{equation*}
  \sup_{x\in \R^d} \|K(x,\cdot)\|_{L^r(\R^d)},\quad 
\sup_{y\in \R^d} \|K(\cdot,y)\|_{L^r(\R^d)}\le C,
\end{equation*}
then $\cT$ is continuous from $L^p(\R^d)$ to $L^q(\R^d)$, 
\begin{equation*}
  \|\cT\phi\|_{L^q(\R^d)}\le C \|\phi\|_{L^p(\R^d)},\quad \forall
  \phi\in \Sch(\R^d). 
\end{equation*}
\end{proposition}
\begin{proof}
  We give a direct proof, as it is simple. If $r$ is finite, setting
  \begin{equation*}
    \alpha=r\(1-\frac{1}{p}\),\quad \beta = \frac{r}{q},
  \end{equation*}
  we have $\alpha+\beta=1$, $\alpha,\beta\ge 0$, and
  H\"older inequality yields, for $\phi\in \Sch(\R^d)$,
  \begin{align*}
  \|\cT \phi\|_{L^q}^q&\le \int\(\int |K(x,y)|
  |\phi(y)|dy\)^qdx  = \int\(\int |K(x,y)|^{\alpha+\beta}
   |\phi(y)|dy\)^qdx\\
    & \le \|\phi\|_{L^p}^q\int \sup_{y\in \R^d}|K(x,y)|^{\beta q} \(
      \int |K(x,y)|^{\alpha p'}dy\)^{q/p'}dx.
  \end{align*}
  As $\beta q = \alpha p'=r$, the assumption yields
  \begin{equation*}
   \|\cT\phi\|_{L^q}^q\lesssim C^{q}\|\phi\|_{L^p}^q . 
 \end{equation*}
 and we note that
 \begin{equation*}
   (r-1)\(1+\frac{q}{p'}\) = (r-1)\( 1+q\(\frac{1}{r}-\frac{1}{q}\)\)
   = q\(1-\frac{1}{r}\) = q\(\frac{1}{p}-\frac{1}{q}\). 
 \end{equation*}
 The case $r=\infty$ corresponds to the situation where $p=1$ and
 $q=\infty$, and the proposition is straightforward.
\end{proof}
\subsection{Functional calculus}
\label{sec:funct}

We recall properties presented in e.g. \cite[Chapter~VIII]{ReedSimon1},
or \cite[Chapter~8]{CheverryRaymond2021}.
The standard way to define functions of a self-adjoint operator (typically,
$H$) consists in using the spectral measure (see
e.g. \cite{ReedSimon1}), or, equivalently Helffer-Sj\"ostrand formula (see
e.g. \cite{DimassiSjostrand}). We note that $H$ is self-adjoint on
$L^2(\R^d)$, it is elliptic, and its spectrum $\si(H)$ is included in
$[1,+\infty)$, 
since $V\ge 1$ by Assumption~\ref{hyp:V}. 
The square root of $H$ is defined in terms of quadratic forms, by
\begin{equation*}
  \<H\phi,\phi\> =\|H^{1/2}\phi\|_{L^2}^2,\quad \forall \phi\in \Sch(\R^d),
\end{equation*}
and we have
\begin{equation*}
  [H^{1/2},H]=0,
\end{equation*}
where $[A,B]$ denotes  the Lie commutator $[A,B]=AB-BA$.
We can also infer the
following result from \cite[Theorem~VIII.5]{ReedSimon1} and
\cite[Propositions~8.3 and 8.20]{CheverryRaymond2021}: 

\begin{proposition}\label{prop:funct-calc}
  Let $V$ satisfying Assumption~\ref{hyp:V}, and $H=-\Delta
  +V$. 
  \begin{itemize}
  \item For any $f\in \Sch(\R)$, $[f(H),H]=0$.
  \item For any $f,g\in \Sch(\R)$, $ f(H)g(H)=(fg)(H) =g(H)f(H)$, on
    \begin{equation*}
      \{u\in \operatorname{Dom}(g(H))\ ;\ g(H) u \in
      \operatorname{Dom}(f(H)) \}.
    \end{equation*}
  \item   For any nonnegative bounded Borelian function $f$, $\|f(H)\|_{L^2\to L^2}\le
    \|f\|_{L^\infty}$.
  \item If $H\psi=\lambda\psi$, then for any $f\in \Sch(\R)$,
    $f(H)\psi=f(\lambda)\psi$.   
  \end{itemize}
\end{proposition}
Recall that $\Pi_\lambda$ is defined by \eqref{eq:Pi_lambda}, with
$\chi\in C_c^\infty(\R^d;[0,1])$.
As a consequence of the first point of the above proposition,
$[\Pi_\lambda,H]=0$, and as 
$[\Pi_\lambda,\partial_t]=0$ since $\Pi_\lambda$ does not depend on time,
$[\Pi_\lambda,e^{-itH}]=0$, hence \eqref{eq:S_lambda}. We insist on
this aspect, since if we had defined initially $\Pi_\lambda$ as a
pseudodifferential operator of symbol
$\chi^2\((|\xi|^2+V(x))/\lambda\)$ (the definition is recalled
below, see \eqref{eq:PDO}), the commutation relation
$[\Pi_\lambda,H]=0$ would have become false, this bracket being only,
in general, a (nontrivial) smoothing operator. It seems crucial,
however, to know that $\Pi_\lambda$ is a pseudodifferential operator,
but  whose symbol is not fully explicit, see Section~\ref{sec:PiPDO}.

\subsection{Weyl-H\"ormander pseudodifferential calculus}
\label{sec:pseudo}
We present aspects of pseudodifferential calculus which can
be found in \cite[Chapter~XVIII]{Hormander3} and
\cite[Chapter~4]{HelfferNier2005}. This  will 
be crucial to establish properties of the spectral localization
$\Pi_\lambda$. 
Define the function $m$ by
 \begin{equation*}
    m(x,\xi) = \sqrt{|\xi|^2+V(x)}. 
  \end{equation*}
We introduce two metrics,
\begin{align*}
&  g_0(x,\xi) = dx^2+d\xi^2,\quad \text{ or, equivalently,}\quad
                 g_{0,(x,\xi)}(y,\eta) = |y|^2+|\eta|^2,\\
  & g_1(x,\xi) =dx^2+\frac{d\xi^2}{|\xi|^2+V(x)}=
    dx^2+\frac{d\xi^2}{m^2(x,\xi)},\quad \text{ or, equivalently,}\\
  &\quad  g_{1,(x,\xi)}(y,\eta) = |y|^2+\frac{|\eta|^2}{m^2(x,\xi)}.                                \end{align*}
We recall \cite[Definitions~18.4.1 and 18.5.1]{Hormander3} (the standard notation
$\si$ in 
the context of pseudodifferential calculus being used only in this
subsection, no confusion with the exponent of the nonlinearity in
\eqref{eq:NLSP} should be possible): 
\begin{definition}
  The metric $g$ is called $\si$ temperate if it is slowly varying,
  \begin{equation*}
    \exists c,C>0,\quad g_{(x,\xi)}(y,\eta)\le c\Longrightarrow
    g_{(x+y,\xi+\eta)}(z,\zeta)\le C g_{(x,\xi)}(z,\zeta),
  \end{equation*}
  and there exist $C,N>0$ such that for all $x_1,x,y,\xi_1,\xi,\eta\in \R^d$,
  \begin{equation*}
    g_{(x_1,\xi_1)}(y,\eta)\le C
    g_{(x,\xi)}(y,\eta)\(1+g^\si_{(x_1,\xi_1)}(x_1-x,\xi_1-\xi)^N\),
  \end{equation*}
  where
  \begin{equation*}
    g^\si_{(x,\xi)}(z\,\zeta) = \sup_{(y,\eta)\not =0}
    \frac{\si(y,\eta;z,\zeta)}{g_{(x,\xi)}(y,\eta)},
    \quad \text{where}\quad  \si(y,\eta;z,\zeta)=y\cdot 
      \zeta-z\cdot \eta.
  \end{equation*}
  A positive function $w$ is $\si$, $g$  temperate
  if it is $g$ continuous,
  \begin{equation*}
    \exists c,C>0,\quad g_{(x,\xi)}(y,\eta)<c\Longrightarrow
    \frac{w(x,\xi)}{C}\le w(x+y,\xi+\eta)\le C w(x,\xi),
  \end{equation*}
  and 
 there exist $C,N>0$ such that
\begin{equation}\label{eq:temperate}
  w(y,\eta)\le C w(x,\xi) \(1+|x-y|^2+|\xi-\eta|^2\)^N,\quad \forall
  x,y,\xi,\eta\in \R^d.
\end{equation}
\end{definition}
We note that, classically (see e.g. \cite{BCM08,HelfferNier2005}),
$g^\si_0=g_0$, and 
\begin{equation*}
  g_1^\si (x,\xi) = m^2(x,\xi)dx^2+d\xi^2. 
\end{equation*}
\begin{lemma}\label{lem:m-temperate}
  Let $m^2=|\xi|^2+V(x)$, where $V$ satisfies
  Assumption~\ref{hyp:V}. Then $m^2$ is $\si$, $g$ temperate, for
  $g=g_0$ and $g=g_1$. 
\end{lemma}
\begin{proof}
The term $|\xi|^2$ is standard, and Peetre inequality (see
e.g. \cite{AlGe07}) yields
\begin{equation*}
  |\eta|^2\le (1+|\xi|^2)(1+|\xi-\eta|^2)\le m^2(x,\xi))(1+|\xi-\eta|^2),
\end{equation*}
where we have used $V\ge 1$ for the last inequality. Regarding the
term $V$, we argue like in the proof of Lemma~\ref{lem:4.1ZAMP}, and
write that from Taylor formula and Assumption~\ref{hyp:V},
\begin{align*}
  V(y)\le V(x) + |y-x||\nabla V(x)| + C|y-x|^2&\lesssim V(x)+
  |y-x|\sqrt{V(x)}+ |x-y|^2 \\
&\lesssim V(x)+ |x-y|^2, 
\end{align*}
where we have used Lemma~\ref{lem:4.1ZAMP} for the second inequality,
and Young inequality for the last one. Then \eqref{eq:temperate}
follows, with $N=1$, since $V\ge 1$. 
\end{proof}

The class of symbols $S(w,g)$ is defined as follows:
\begin{definition}
  [Symbols associated with a weight and a metric]\label{def:symbols-g}
 Let $w$ be $\si$ temperate.
 The set $S(w,g)$ consists of functions $a\in C^\infty(\R^d\times
  \R^d)$ such that for every $x,\xi\in \R^d$, every $\alpha,\beta\in
  \N^d$,
  \begin{equation*}
    |\partial_x^\alpha\partial_\xi^\beta a(x,\xi)|\le C(\alpha,\beta)
    \left\{
      \begin{aligned}
        &w(x,\xi)\quad \text{if }g=g_0,\\
        &w(x,\xi)m(x,\xi)^{-|\beta|}\quad \text{if }g=g_1.
      \end{aligned}
    \right.
  \end{equation*}
The best constants $C(\alpha,\beta)$ define the seminorms of $a$,
\begin{equation*}
  |a|_{g_j,\alpha,\beta} = \sup_{(x,\xi)\in \R^d\times\R^d}
w(x,\xi)^{-1} m(x,\xi)^{j|\beta|}  |\partial_x^\alpha\partial_\xi^\beta a(x,\xi)|
,\quad j=0,1.
\end{equation*}
\end{definition}
We consider the Weyl quantization for pseudodifferential
operators. For a symbol $a\in S(w,g)$ and $\phi\in \Sch(\R^d)$, we
set, with the usual convention $D=-i\nabla$,
\begin{equation}\label{eq:PDO}
  a^w(x,D)\phi = (\Op a)\phi(x) = (2\pi)^{-d}\iint_{\R^{2d}}
  e^{i(x-y)\cdot\xi}a\(\frac{x+y}{2},\xi\)\phi(y)dyd\xi.
\end{equation}
We can also express $\Op a$ thanks to its Schwartz kernel,
\begin{equation*}
    \(\Op a \)\phi(x) =\int_{\R^d}K(x,y)\phi(y)dy,\quad K(x,y) :=
    (2\pi)^{-d}\int_{\R^d} 
  e^{i(x-y)\cdot \xi }    a\(\frac{x+y}{2},\xi\)d\xi.
\end{equation*}
We denote
\begin{equation*}
  OPS(w,g) =\{ a^w,\quad a\in S(w,g)\}. 
\end{equation*}
The composition of pseudodifferential operators is described in
 \cite[Theorem~18.5.4]{Hormander3}, and implies the following
result, where we consider only the metric $g_1$:
\begin{proposition}\label{prop:compositionPDO}
  If $a_1\in S(w_1,g_1)$ and $a_2\in S(w_2,g_1)$, then $a_1^w\circ
  a_2^w\in OPS(w_1w_2,g_1)$. The symbol of this pseudodifferential
  operator is denoted by $a_1\sharp a_2$, and we have, for all $N\in \N$,
  \begin{align*}
    a_1\sharp a_2 (x,\xi)-&
    \sum_{j=0}^{N-1}\frac{1}{j!}\(\frac{i\si(D_x,D_\xi;D_y,D_\eta)}{2}\)^j
    a_1(x,\xi)a_2(y,\eta)\Big|_{(y,\eta)=(x,\xi)} =r_N(x,\xi),\\
    &\text{where} \quad r_N\in
    S(m^{-N}w_1w_2,g_1). 
  \end{align*}
\end{proposition}
We note that $H=\pp^w$ for
\begin{equation*}
  \pp(x,\xi) = |\xi|^2+V(x)\in S(m^2,g_1). 
\end{equation*}
As $V\ge 1$, $H$ is elliptic (since $\pp(x,\xi)=
m^2(x,\xi)$), and we
emphasize two consequences:
\begin{proposition}[From Theorem~4.8 in
  \cite{HelfferNier2005}]\label{prop:sqrtH} 
  The operator $\sqrt H$ is a pseudodifferential operator, belonging
  to $OPS(m,g_1)$, and 
  \begin{equation*}
    \sqrt H - (\sqrt \pp)^w\in OPS(1,g_1).
  \end{equation*}
\end{proposition}
Together with Proposition~\ref{prop:compositionPDO}, and the proof of
\cite[Theorem~18.1.9]{Hormander3}, this implies: 
\begin{proposition}\label{prop:elliptic}
There exists $\qq\in S(m^{-2},g_1)$, with $\sqrt \qq\in S(m^{-1},g_1)$,
such that for all $N\in \N$,  
\begin{align*}
  &H\circ (\qq)^w-{\rm I},  \  (\qq)^w\circ H-{\rm I}\in
    OPS(m^{-N},g_1),\\
  &\sqrt H\circ (\sqrt \qq)^w-{\rm I},  \  (\sqrt \qq)^w\circ \sqrt
    H-{\rm I}\in 
    OPS(m^{-N},g_1).
\end{align*}
More precisely,
\begin{equation*}
  \qq(x,\xi) -\frac{1}{\pp(x,\xi)} \in  S(m^{-3},g_1),\quad
  \sqrt{\qq(x,\xi) }-\frac{1}{\sqrt{ \pp(x,\xi)}} \in  S(m^{-2},g_1).
\end{equation*}
 \end{proposition}
The metric $g_1$ satisfies the assumption of \cite{Beals1979}, and we
have, from \cite[Theorem~3.1]{Beals1979}:
\begin{proposition}\label{prop:Beals}
 Let $1<p<\infty$.  If $a\in S(1,g_1)$, then $a^w$ is bounded in
 $L^p$, and its operator norm is controlled by finitely many seminorms
 of $a$. 
\end{proposition}
\begin{remark}
  In the case $V=0$, the set $S(1,g_1)$ coincides with the class of
  symbols $S^0$ from \cite[Definition~18.1.1]{Hormander3}, and the
  above result meets e.g. \cite[Chapter~VI, Section~5.2]{Stein93}, 
\cite[Theorem~5.2]{Taylor3}, or
\cite[Theorem~3.1.6]{Sogge2017}. The proof is actually similar in our
case, based on decay estimates for the Schwartz kernel $K$ off the
diagonal (obtained by integrations by parts, using the property
$m(x,\xi)\ge |\xi|$), which makes it possible to eventually invoke
Proposition~\ref{prop:schur} (with $r=1$).
\end{remark}

As a first application of these general results, we compare Sobolev norms:
\begin{lemma}[Equivalence of norms]\label{lem:sobolev-H}
  Let $V$ satisfying Assumption~\ref{hyp:V}. For any $1< p <\infty$,
  there exist $C_1,C_2>0$ such that for any 
$\phi\in\Sch(\R^d)$,
\begin{align*}
  &\|\nabla \phi\|_{L^p}+ \|\phi\sqrt V\|_{L^p}\le C_1
  \|H^{1/2} \phi\|_{L^p}\le C_2 
  \(\|\nabla \phi\|_{L^p}+ \|\phi\sqrt V\|_{L^p}\),\\
&\|\Delta \phi\|_{L^p}+ \|V \phi \|_{L^p}\le C_1
  \|H \phi\|_{L^p}\le C_2 
  \(\|\Delta \phi\|_{L^p}+ \|V \phi\|_{L^p}\).
\end{align*}
\end{lemma}
\begin{proof}
We follow the same strategy as in the proof of \cite[Lemma~2.7]{KiViZh09}, and
invoke analytic interpolation (from \cite[Section~V.4]{SteinWeiss},
see also \cite[Theorem~2.7]{LiPo09}). 
  First, considering the function on $\R_+\setminus \{0\}$, $x\mapsto
  x^{is}$ for $s\in \R$, the main result from \cite{Hebisch1990}
  implies the bound
  \begin{equation*}
    \|H^{is}\phi\|_{L^p} \lesssim \|\phi\|_{L^p}.
  \end{equation*}
The estimate $\|V^{is}\phi\|_{L^p}\lesssim \|\phi\|_{L^p}$ is
straightforward, and $\| (-\Delta)^{is}\phi\|_{L^P}\lesssim
\|\phi\|_{L^p}$ follows from the H\"ormander-Mikhlin multiplier theorem
(see e.g. \cite[Theorem~7.9.5]{Hormander1}, or \cite[Theorem~2.8]{LiPo09}), so
 \begin{equation*}
 \|V^{is}\phi\|_{L^p}+  \| (-\Delta)^{is}\phi\|_{L^p}
 +\|H^{is}\phi\|_{L^p} \lesssim \|\phi\|_{L^p},\quad \forall s\in \R.
  \end{equation*}
On the other hand, triangle inequality yields
\begin{equation*}
  \|H \phi\|_{L^p}\le
  \|\Delta \phi\|_{L^p}+ \|V \phi\|_{L^p}.
\end{equation*}
The maps
\begin{equation*}
  (x,\xi)\mapsto \frac{|\xi|^2}{|\xi|^2+V(x) }\quad\text{and}\quad
  (x,\xi)\mapsto \frac{V(x)}{|\xi|^2+V(x) } 
\end{equation*}
define symbols, as can be checked
directly, keeping Lemma~\ref{lem:4.1ZAMP} in mind for the second
map. Proposition~\ref{prop:Beals} implies   
\begin{equation*}
\|\Delta H^{-1}\phi\|_{L^p}+ \|V H^{-1}\phi\|_{L^p}\lesssim \|\phi\|_{L^p},
\end{equation*}
hence the second inequality of the lemma by replacing $\phi$ with
$H\phi$. 
Analytic
interpolation for the operator-valued functions $z\mapsto
(-\Delta)^zH^{-z}$ and $V^zH^{-z}$ yields
\begin{equation*}
  \|(-\Delta)^s \phi\|_{L^p} + \|V^s\phi\|_{L^p}\lesssim
  \|H^s \phi\|_{L^p}\lesssim \|(-\Delta)^s \phi\|_{L^p} +
  \|V^s\phi\|_{L^p} ,\quad \text{for all }0\le s\le 1,
\end{equation*}
as the case $s=0$ is trivial. The first inequality of the lemma
corresponds to the case $s=1/2$.
\end{proof}

\begin{remark}
We  give an alternative proof of the main estimate in the proof of
Lemma~\ref{lem:sobolev-H}, in the 
spirit of \cite[Lemma~2.4]{YajZha04}:
\begin{equation}\label{eq:ineq-sob-H}
  \|\Delta \phi\|_{L^p}+ \|V \phi \|_{L^p}\lesssim
  \|H \phi\|_{L^p} .
\end{equation}
Setting $Q=\qq^w$, Proposition~\ref{prop:elliptic} yields
\begin{equation*}
  \Delta \phi = \Delta QH\phi  + \Delta R \phi,\quad V \phi = V QH\phi
  + V R\phi,\quad R \in OPS (m^{-2},g_1) .
\end{equation*}
Since $VQ$
and $\Delta \circ Q$ belong to $OPS (1,g_1)$,  Proposition~\ref{prop:Beals}
implies
\begin{equation*}
  \|\Delta  QH\phi\|_{L^p}+ \|V QH\phi \|_{L^p}\lesssim
  \|H \phi\|_{L^p} .
\end{equation*}
Also, $\Delta \circ R,VR\in  OPS (1,g_1)$, hence
\begin{equation*}
 \|\Delta R\phi\|_{L^p}+ \|V R\phi \|_{L^p}\lesssim
  \|\phi\|_{L^p}\lesssim \|H\phi\|_{L^p},
\end{equation*}
since $H^{-1}\in OPS (m^{-2},g_1)$, and using Proposition~\ref{prop:Beals}
again, hence
 \eqref{eq:ineq-sob-H}. Note that by considering
 Proposition~\ref{prop:elliptic} again (now for the parametrix of $\sqrt
 H$), we also have, by 
following the same lines as above,
\begin{equation*}
  \|\nabla \phi\|_{L^p}+ \| \phi \sqrt V\|_{L^p}\lesssim
  \|\sqrt H \phi\|_{L^p} ,
\end{equation*}
without invoking interpolation.
\end{remark}

The following consequences of Lemma~\ref{lem:sobolev-H} will be
convenient: if $1<p<\infty$, $f,g\in \Sch(\R^d)$, and $\gamma>0$,
\begin{equation}\label{eq:H12der}
\begin{aligned}
 & \|H^{1/2}(fg)\|_{L^p}\lesssim \|f\|_{L^{a_1}}\|H^{1/2}g\|_{L^{a_2}}+
 \|\nabla f\|_{L^{a_3}}\|g\|_{L^{a_4}}\\
 &\phantom{\|H^{1/2}(fg)\|_{L^p}} \lesssim
  \|f\|_{L^{a_1}}\|H^{1/2}g\|_{L^{a_2}}+ 
  \|H^{1/2} f\|_{L^{a_3}}\|g\|_{L^{a_4}},\\
&\qquad  \frac{1}{p}=\frac{1}{a_1}+\frac{1}{a_2}=
  \frac{1}{a_3}+\frac{1}{a_4},\\
&  \|H^{1/2}\(|f|^{\gamma}f\)\|_{L^p}\lesssim \|f\|_{L^{a_1}}^{\gamma}
  \|H^{1/2}f\|_{L^{a_2}},\quad  \frac{1}{p}=\frac{\gamma}{a_1}+\frac{1}{a_2}.
\end{aligned}
\end{equation}

\subsection{Semiclassical pseudodifferential calculus}

The spectral localization $\Pi_\lambda$ is defined by functional
calculus, in order to have the exact commutation relation
$[\Pi_\lambda,H]=0$.  In the case $V=0$, $\Pi_\lambda$ is a Fourier
multiplier, which makes it possible to establish various estimates,
needed for the analysis of error estimates for the discretization in
time. When $V\not =0$, the generalization of the notion of Fourier
multiplier is the notion of pseudodifferential operator. It is
therefore crucial to know that $\Pi_\lambda$ is indeed a
pseudodifferential operator: this follows from the initial result of
Helffer and Robert \cite{HelfferRobert1983}, which was revisited by
Dimassi and Sj\"ostrand \cite{DimassiSjostrand}.
For $\hbar\in (0,1]$, the semiclassical version of \eqref{eq:PDO}
is 
\begin{equation}\label{eq:pseudo-h}
 a^w(x,\hbar D) \phi=\(\Op_\hbar a \)\phi(x) =
  (2\pi\hbar)^{-d}\iint_{\R^{2d}} 
  e^{i(x-y)\cdot \xi /\hbar}    a\(\frac{x+y}{2},\xi\)\phi(y)dyd\xi.
\end{equation}
In view of Lemma~\ref{lem:m-temperate}, we have, from
\cite[Theorem~4.1]{HelfferRobert1983} or
  \cite[Theorem~8.7]{DimassiSjostrand}:

\begin{proposition}\label{prop:HelfferRobert}
 Let $P(\hbar)= \Op_\hbar \pp$,
 where $\pp\in S(m^2,g_0)$ is real-valued.
 %, and possesses an
 %asymptotic expansion as $\hbar\to 0$,
 %\begin{equation*}
 %  p^\hbar (x,\xi) = \sum_{j=0}^n \hbar^jp_j(x,\xi) +h^{n+1}r_{n+1}^\hbar(x,\xi),
 %\end{equation*}
% where for all  $n\ge 0$, $r_{n+1}^\hbar\in S(m,g_0)$ uniformly in
% $\hbar \in (0,1]$
 If $f\in C_c^\infty(\R)$, then $f(P)\in \Op_\hbar
 (S(m^{-k},g_0))$ for every $k\in \N$.  In addition,
 \begin{equation*}
   f(P) = \Op_\hbar \(a_0 + \hbar a_1+\hbar^2 \rho^\hbar\),
 \end{equation*}
 where
 \begin{equation*}
   a_0(x,\xi)=  f\(\pp(x,\xi)\),\quad a_1(x,\xi) = \pp_1(x,\xi)f'\(\pp(x,\xi)\),
 \end{equation*}
 for $\pp_1\in S(m^2,g_0)$ and $\rho^\hbar\in S(m^{-k},g_0)$ uniformly in
 $\hbar\in (0,1]$   (the corresponding constants $C(\alpha,\beta)$ from
 Definition~\ref{def:symbols-g}  do not depend on $\hbar\in (0,1]$).
\end{proposition}

\subsection{The spectral cutoff as a pseudodifferential operator}
\label{sec:PiPDO}
We can now prove that $\Pi_\lambda$, defined by functional calculus,
is indeed a pseudodifferential operator, whose kernel is estimated
like in the case $V=0$:

\begin{proposition}\label{prop:Pi-pseudo}
  Let $\chi\in C_c^\infty(\R;[0,1])$, equal to one on $[-1,1]$, and
  zero outside $[-2,2]$, and $\Pi_\lambda$ defined by
  \eqref{eq:Pi_lambda}, where $V$ satisfies Assumption~\ref{hyp:V}.
  \begin{itemize}
  \item  For every $\lambda>0$, $\Pi_\lambda $ is a pseudodifferential
    operator, $\Pi_\lambda\in OPS(1,g_1)$.
  \item Its Schwartz kernel $K_\lambda$ is such that
    \begin{equation*}
      \Pi_\lambda \phi(x)  = \int_{\R^d}K_\lambda(x,y)\phi(y)dy,\quad
      \forall \phi\in \Sch(\R^d),
    \end{equation*}
    and for all $N>0$, there exists $C_N>0$ independent of $\lambda\ge
    1$ such that 
    \begin{equation}\label{eq:est-K-lambda}
      |K_\lambda(x,y)|\le C_N\frac{\lambda^{d/2}}{\(1+ \sqrt\lambda|x-y|\)^N}.
    \end{equation}
    \item We have the following decomposition:
      \begin{equation*}
        \Pi_\lambda  = \pi_\lambda^w,\quad \pi_\lambda = a_\lambda
        +\frac{1}{\sqrt\lambda} b_\lambda + \frac{1}{\lambda} r_\lambda,
      \end{equation*}
      where
      \begin{equation*}
        a_\lambda(x,\xi) = \chi^2 \(\frac{|\xi|^2+V(x)}{\lambda}\),\quad
        b_\lambda(x,\xi) = \qq_1(x,\xi)\chi \chi
        '\(\frac{|\xi|^2+V(x)}{\lambda}\), 
      \end{equation*}
      for some $\qq_1\in C^\infty(\R^d\times \R^d)$ independent of
      $\lambda$, and $r_\lambda\in \cap_k S(m^{-k},g_0)$ whose
      seminorms are bounded uniformly in $\lambda\ge 1$. In particular,
      \begin{equation*}
        \forall \alpha,\beta\in \N^d,\quad \exists
        C(\alpha,\beta),\quad |\partial_x^\alpha\partial_\xi^\beta
        r_\lambda(x,\xi)|\le C(\alpha,\beta)m(x,\xi)^{-|\beta|},\quad
        \forall \lambda \ge 1.
      \end{equation*}
  \end{itemize}
\end{proposition}
\begin{proof}
  Set $\hbar=1/\sqrt\lambda$, and $p_\hbar=|\xi|^2+ \hbar^2 V$. Then
  \begin{equation*}
    \frac{H}{\lambda} = -\hbar^2\Delta +\hbar^2V(x) = \Op_\hbar
   p_\hbar= \hbar^2 \pp^w,
  \end{equation*}
  and we have $p_\hbar \in S(m^2,g_0)$. We can therefore invoke
  Proposition~\ref{prop:HelfferRobert}, with $f=\chi^2$:
  \begin{equation*}
    \Pi_\lambda  = \Op_\hbar \alpha^\hbar,\quad \alpha^\hbar =
    \alpha_0 + \frac{1}{\sqrt\lambda} \alpha_1 + 
    \frac{1}{\lambda}\rho^\hbar,
  \end{equation*}
  with
  \begin{equation*}
    \alpha_0(x,\xi) = \chi^2\(|\xi|^2+ \hbar^2 V(x)\),\quad
    \alpha_1(x,\xi) = \pp_1(x,\xi) \chi\chi' \(|\xi|^2+ \hbar^2 V(x)\),
  \end{equation*}
  and $\pp_1\in S(m^2,g_0)$, $\rho^\hbar\in S(m^{-k},g_0)$ uniformly in $\hbar \in (0,1]$ for
  all $k\in \N$. The first claim of the proposition follows
  readily. In addition, we have
  \begin{align*}
    K_\lambda(x,y)& = (2\pi\hbar)^{-d}\int_{\R^d} e^{i(x-y)\cdot
      \xi/\hbar}\alpha^\hbar\(\frac{x+y}{2},\xi\)d\xi\\
&=\(\frac{\sqrt\lambda}{2\pi}\)^d
      \int_{\R^d}  e^{i\sqrt\lambda (x-y)\cdot 
      \xi}\alpha^\hbar\(\frac{x+y}{2},\xi\)d\xi.
  \end{align*}
  As $\alpha^\hbar$ is integrable in $\xi$, uniformly in $x\in \R^d$
  and $\hbar\in (0,1]$, $|K_\lambda(x,y)|\lesssim
  \hbar^{-d}=\lambda^{d/2}$. Since for all $\beta\in \N^d$, $\partial_\xi^\beta
  \alpha^\hbar$ is also integrable in $\xi$, uniformly in $x\in \R^d$
  and $\hbar\in (0,1]$, $N$ integrations by parts yield
  \begin{equation*}
    |K_\lambda(x,y)|\lesssim \frac{\lambda^{d/2}}{(\sqrt\lambda|x-y|)^N},
  \end{equation*}
  hence \eqref{eq:est-K-lambda} by combining these two estimates.

  Going back to the definitions \eqref{eq:PDO} and
  \eqref{eq:pseudo-h}, and changing variables in $\xi$, we have
  \begin{equation*}
    \Op_\hbar \alpha_0 =  a_\lambda^w,\ 
\Op_\hbar \alpha_1=
    b_\lambda^w + 
    \hbar \(r_1^\hbar\)^w, \  \Op_\hbar \rho^\hbar = \(r_2^\hbar\)^w,
  \end{equation*}
  where $r_1^\hbar,r_2^\hbar$ satisfy
  \begin{equation*}
    \forall \alpha,\beta\in \N^d, \ \forall k\in \N,
    \quad |\partial_x^\alpha\partial_\xi^\beta r_j^\hbar (x,\xi)|\le C
    m(x,\xi)^{-k},\quad \forall x,\xi\in \R^d,\ \forall \hbar \in (0,1].
  \end{equation*}
  The proposition follows by setting $r =r_1+r_2$.
 \end{proof}
\subsection{Estimates on the spectral localization}
\label{sec:spectral}

We can now establish some results which are the analogue of the tools
used in \cite{ChoiKoh2021,Ignat2011} for the case $V=0$, where
$\Pi_\lambda$ is a Fourier multiplier. 
\begin{lemma}\label{lem:PiPDO}
For any $1< p <\infty$, there exists $C$ such that for any
$\phi\in\Sch(\R^d)$ and $\lambda\ge 1$, 
   \begin{equation}\label{eq:Pi-lambda-Lp}
    \| \Pi_\lambda \phi\|_{L^p }
    \le C \|\phi\|_{L^p} ,
    \end{equation}
    \begin{equation}\label{eq:Pi-lambda-1}
    \left\| \Pi_\lambda\phi-\phi \right\|_{L^{p} }
    \le \frac{C}{\sqrt\lambda}\|H^{1/2} \phi \|_{L^{p} }.
    \end{equation}
 \begin{equation}\label{eq:Pi-lambda-1bis}
    \left\| \Pi_\lambda\phi-\phi \right\|_{L^{p} }
    \le \frac{C}{\lambda}\|H \phi \|_{L^{p} }.
    \end{equation}
 \end{lemma}
 \begin{remark}
   If we consider only the case $p=2$, the above lemma is a direct
   consequence of Proposition~\ref{prop:funct-calc}, since $\chi$ is
   bounded, as well as 
   \begin{equation*}
     g_j(z) = \frac{\chi^2(z/\lambda)-1}{z^{j/2}} =
     \frac{1}{\lambda^{j/2}}\frac{\chi^2(\zeta)-1}{\zeta^{j/2}}
     \Big|_{\zeta=z/\lambda} ,\quad j=1,2,
   \end{equation*}
is bounded, with $\|g_j\|_{L^\infty}\lesssim \lambda^{-j/2}$. However,
we will need the above inequalities in the case $p\not =2$, for which
pseudodifferential calculus turns out to be very helpful. 
 \end{remark}
 \begin{proof}
   The first inequality is a direct consequence of
   \eqref{eq:est-K-lambda} and Proposition~\ref{prop:schur}. 
\smallbreak

To prove \eqref{eq:Pi-lambda-1} and \eqref{eq:Pi-lambda-1bis}, we invoke
Proposition~\ref{prop:elliptic} (with $N=2$): if $\beta=1$ or $2$,
\begin{equation*}
  \Pi_\lambda -1 = \frac{1}{\lambda^{\beta/2}} \lambda^{\beta/2}\(
  \(\Pi_\lambda-1\) \(\qq^{\beta/2}\)^w  H^{\beta/2} +
  \(\Pi_\lambda-1\) R_\beta\),  
\end{equation*}
with $R_\beta\in OPS(m^{-2},g_1)$. 
In view of our assumption our $\chi$, the symbol $\pi_\lambda$, given in
Proposition~\ref{prop:Pi-pseudo}, satisfies:
\begin{equation*}
  \lambda^{\beta/2}\(\pi_\lambda-1\)=
  \(\frac{\sqrt\lambda}{m}\)^\beta\(\pi_\lambda-1\) m^\beta  
\in S(m^{\beta},g_1), \quad
  \text{uniformly in }\lambda\ge 1. 
\end{equation*}
Therefore, Proposition~\ref{prop:compositionPDO} implies that 
\begin{equation*}
  \lambda^{\beta/2}
  \(\Pi_\lambda-1\) \(\qq^{\beta/2}\)^w \text{ and }\lambda^{\beta/2} 
  \(\Pi_\lambda-1\) R_\beta
\end{equation*}
are pseudodifferential in $OPS(1,g_1)$, whose symbols have seminorms
uniformly bounded for $\lambda\ge 1$. Proposition~\ref{prop:Beals}
then concludes the proof. 
\end{proof}

\begin{lemma}\label{lem:perte}
   For any $1< p<\infty$, there exists $C$ such that for any
   $\phi\in\Sch(\R^d)$ and any $\lambda\ge 1$,
  \begin{equation*}
    \left\|  \Pi_\lambda  H^{1/2}\phi\right\|_{L^p} \le C \sqrt\lambda
    \|\phi\|_{L^p}. 
  \end{equation*}
\end{lemma}
\begin{proof}
  In view of Proposition~\ref{prop:sqrtH}, $H^{1/2}\in OPS(m,g_1)$,
  and from Proposition~\ref{prop:Pi-pseudo}, $\Pi_\lambda\in
  OPS(1,g_1)$, so Proposition~\ref{prop:compositionPDO} implies that
  the operator $\Pi_\lambda  H^{1/2}= H^{1/2}\Pi_\lambda$ is a
 pseudodifferential operator, $\Pi_\lambda  H^{1/2}\in
 OPS(m,g_1)$. Propositions~\ref{prop:compositionPDO}, \ref{prop:sqrtH}
 and \ref{prop:Pi-pseudo} also entail that the symbol of $\Pi_\lambda
 H^{1/2}$ is 
 \begin{equation*}
\sqrt \pp \chi^2\(\frac{\pp}{\lambda}\) +\qq_1 \sqrt{\frac{\pp}{\lambda}}
\chi\chi'\(\frac{\pp}{\lambda}\) + \frac{1}{\lambda}\rho_\lambda,
\end{equation*}
where $\rho_\lambda \in S(1,g_1)$ has its seminorms bounded uniformly
in $\lambda\ge 1$. Writing
\begin{equation*}
\sqrt \pp \chi^2\(\frac{\pp}{\lambda}\)=
 \sqrt\lambda \sqrt{\frac{\pp}{\lambda}} \chi^2\(\frac{\pp}{\lambda}\) ,
\end{equation*}
Proposition~\ref{prop:Beals} yields the result. 
\end{proof}

We conclude this section with an easy generalization of Bernstein inequality:
\begin{lemma}\label{lem:bernstein}
   For any $1\le  p\le q\le \infty$, there exists $C$ such that for any
   $\phi\in\Sch(\R^d)$ and any $\lambda\ge 1$,
 \begin{equation*}
 \|\Pi_\lambda \phi\|_{L^q(\R^d)}\le C
 \lambda^{\frac{d}{2}\(\frac{1}{p}-\frac{1}{q}\)} \|\phi\|_{L^p(\R^d)}.
 \end{equation*}
\end{lemma}
\begin{proof}
This follows directly from Proposition~\ref{prop:schur}
and \eqref{eq:est-K-lambda}: if $r$ is finite,
  \begin{equation*}
    \|K_\lambda(x,\cdot)\|_{L^r}^r +
    \|K_\lambda(\cdot,y)\|_{L^r}^r\lesssim
    \int_{\R^d}\frac{\lambda^{rd/2}}{\(1+(\sqrt\lambda|z|)^{d+1}\)^r}dz\lesssim
    \lambda^{(r-1)d/2},
  \end{equation*}
and, if $r=\infty$, $\|K_\lambda\|_{L^\infty_{x,y}}\lesssim
  \lambda^{d/2}$. 
\end{proof}

  \section{Strichartz estimates}
\label{sec:strichartz}

We recall Strichartz inequalities in the case of continuous time, and
prove their analogue in the discrete case, thanks to the operator
$\Pi_\lambda$. 
\subsection{Continuous time}

\begin{definition}\label{def:adm}
 A pair $(q,r)$ is admissible if $2\le r
  <\frac{2d}{d-2}$ ($2\le r\le\infty$ if $d=1$, $2\le r<
  \infty$ if $d=2$)
  and
\[\frac{2}{q}=\delta(r):= d\left( \frac{1}{2}-\frac{1}{r}\right).\]
\end{definition}
\begin{remark}
  We note that the range for $q$ is equivalent to: $q\in (2,\infty]$ if
$d\ge 2$  (we do not consider the
endpoint case), and $q\in [4,\infty]$ if $d=1$. 
\end{remark}

It is a direct consequence of the main result in \cite{Fujiwara} that
the group $S(t)$ satisfies local in time dispersive estimates, in the
sense  that
there exists $\delta>0$ such for all $\phi\in \Sch(\R^d)$,
\begin{equation}\label{eq:disp-cont}
  \|S(t) \phi\|_{L^\infty(\R^d)}\le
  \frac{C}{|t|^{d/2}}\|\phi\|_{L^1(\R^d)},\quad |t|\le \delta. 
\end{equation}
The fact that such estimates are local in time may be sharp under
Assumption~\ref{hyp:V}, even though the restriction $|t|\le \delta$
can be removed when $V\equiv 0$. Typically in the harmonic case
$V(x)=|x|^2$, Hermite functions $\varphi_n$, $n\ge 0$,
satisfy $H\varphi_n= \lambda_n \varphi_n$ for $\lambda_n\in
d+2\N$, and $S(t)\varphi_n(x) = e^{-it\lambda_n}\varphi_n(x)$
has a constant $L^\infty$-norm in space.

Together with the fact that $S(t)$ is unitary on $L^2(\R^d)$, we infer
from \cite{KeelTao} the classical result (see
e.g. \cite[Section~9.2]{CazCourant}):
\begin{proposition}[Strichartz estimates]\label{prop:strichartz}
  Let $d\ge 1$ and $S(t)=e^{-itH}$. \\
$(1)$ \emph{Homogeneous estimates.} For any admissible pair $(q,r)$,
there exists $C_{q}$  such that for any $T>0$,
\begin{equation*}
 % \label{eq:stri-homo}
  \|S(t)\phi\|_{L^q([0,T];L^r)} \le C_q T^{1/q}
\|\phi \|_{L^2},\quad \forall \phi\in L^2.
\end{equation*}
$(2)$ \emph{Inhomogeneous estimates.}
Denote
\begin{equation*}
  D(F)(t,x) = \int_0^t S(t-s)F(s,x)ds.
\end{equation*}
For all admissible pairs $(a_1,b_1)$ and~$
    (a_2,b_2)$, there exists $C=C_{a_1,a_2}$ such that for any bounded
    interval $I\ni 0$ such that $|I|\le 1$,
\begin{equation}\label{eq:strichnl}
      \left\lVert D(F)
      \right\rVert_{L^{a_1}(I;L^{b_1})}\le C \left\lVert
      F\right\rVert_{L^{a'_2}\(I;L^{b'_2}\)},\quad \forall F\in L^{a'_2}(I;L^{b'_2}).
\end{equation}
\end{proposition}

\subsection{Discrete time}

As pointed out in \cite{Ignat2011} in the case $V=0$, discrete in time Strichartz
estimates cannot be true without a frequency cutoff. One way to
understand  this
consists in recalling that the proof of Strichartz estimates relies on
Hardy-Littlewood-Sobolev inequality (fractional integration), and the
discrete version of this result requires some care regarding the
behavior of the kernel at the origin, as integration is replaced by
summation. We give more details below to explain this phenomenon, by
revisiting a result from \cite{SteinWainger2000}: 

\begin{lemma}\label{lem:discHLS}
  Let $0<\alpha<1$, and consider a kernel satisfying, for some $c_0,c_1>0$,
  \begin{equation*}
    |k(n)|\le \frac{1}{c_0+c_1|n|^\alpha},\quad n\in \Z.
  \end{equation*}
  Then the operator $I$ defined by $If=f\ast k$ is continuous from
  $\ell^p(\Z)$ to $\ell^q(\Z)$ provided that $1<p<q<\infty$ and
  \begin{equation*}
    \frac{1}{q}= \frac{1}{p}-1+\alpha.
  \end{equation*}
  There exists $C$ independent of $c_0$ and $c_1$ such that for all
  $f\in \ell^p(\Z)$,
  \begin{equation*}
    \|If\|_{\ell^q(\Z)}\le
      C\max\(\frac{1}{c_0},\frac{1}{c_1}\)\|f\|_{\ell^p(\Z)}. 
  \end{equation*}
\end{lemma}
\begin{proof}
  We single out the diagonal term as follows:
  \begin{align*}
    \|If\|_{\ell^q(\Z)}^q & = \sum_{n\in \Z} |If(n)|^q= \sum_{n\in
                            \Z}\left|\sum_{m\in \Z}    f(n-m)k(m)\right|^q\\
    &\lesssim \sum_{n\in  \Z}\(\left|\sum_{m\in \Z\setminus\{0\}}
      f(n-m)k(m)\right|^q + \left|
      f(n)k(0)\right|^q\) \\
    &\lesssim \sum_{n\in  \Z}\(\left|\sum_{m\in \Z\setminus\{0\}}
      f(n-m)\frac{1}{c_1|m|^\alpha}\right|^q + \frac{1}{c_0^q}\left|
      f(n)\right|^q\) 
  \end{align*}
  The first (double) sum is estimated thanks to
  \cite[Proposition~(a)]{SteinWainger2000}, by
  \begin{equation*}
     \sum_{n\in  \Z}\left|\sum_{m\in \Z\setminus\{0\}}
      f(n-m)\frac{1}{c_1|m|^\alpha}\right|^q\lesssim
    \frac{1}{c_1^q}\|f\|_{\ell^p(\Z)}^q. 
  \end{equation*}
  The second sum is equal to
  \begin{equation*}
    \frac{1}{c_1^q}\|f\|_{\ell^q(\Z)}^q\le \frac{1}{c_1^q}\|f\|_{\ell^p(\Z)}^q,
  \end{equation*}
  since $\ell^p(\Z)\subset \ell^q(\Z)$ from the assumption $p<q$. The
  lemma easily follows. 
\end{proof}
\begin{lemma}\label{lem:disp-disc}
  There exist $C>0$ and $\delta>0$ such that for all $\lambda>0$, $S_\lambda$,
  defined by \eqref{eq:S_lambda}, satisfies:
  \begin{equation*}
    \|S_\lambda(t)\phi\|_{L^\infty(\R^d)}\le
      \frac{C}{\lambda^{-d/2}+|t|^{d/2}}\|\phi\|_{L^1(\R^d)},\quad
      \forall \phi\in \Sch(\R^d),\ |t|\le \delta. 
  \end{equation*}
\end{lemma}
\begin{proof}
  First, we use the dispersive estimate \eqref{eq:disp-cont} from
  \cite{Fujiwara}, and write, for $\phi\in \Sch(\R^d)$ and $|t|\le \delta$,
  \begin{equation*}
    \|S_\lambda(t)\phi\|_{L^\infty(\R^d)} =
    \|S(t)\Pi_\lambda\phi\|_{L^\infty(\R^d)}\lesssim
    \frac{1}{|t|^{d/2}}\|\Pi_\lambda\phi\|_{L^1(\R^d)}. 
  \end{equation*}
  Choosing for instance $N=d+1$ in \eqref{eq:est-K-lambda} yields
 \begin{equation*}
   \|K_\lambda\|_{L^\infty_yL^1_x}\lesssim 1.
 \end{equation*}
 Fubini Theorem implies
 \begin{equation*}
   \|\Pi_\lambda \phi\|_{L^1(\R^d)}\lesssim  \|\phi\|_{L^1(\R^d)}.
 \end{equation*}
 This yields the lemma for $|t|>1/\lambda$. For small values of
 $t$, we write
 \begin{equation*}
S_\lambda(t)=   \underbrace{\chi\(\frac{H}{\lambda}\)}_{L^2\to
    L^\infty}\circ\underbrace{e^{-itH}}_{L^2\to L^2} \circ
\underbrace{\chi\(\frac{H}{\lambda}\)}_{L^1\to L^2},
\end{equation*}
and invoke the classical $TT^*$ argument, which entails, since
$\chi\(\frac{H}{\lambda}\)$ is self-adjoint,
\begin{equation*}
  \left\|\chi\(\frac{H}{\lambda}\)\right\|_{L^2\to L^\infty} =
  \left\|\chi\(\frac{H}{\lambda}\)\right\|_{L^1\to L^2} =   
\left\|\chi^2\(\frac{H}{\lambda}\)\right\|_{L^1\to L^\infty}^{1/2}
= \|\Pi_\lambda\|_{L^1\to L^\infty}^{1/2}, 
\end{equation*}
so we can write, since $e^{-itH}$ is unitary on $L^2$,
\begin{equation*}
  \|S_\lambda(t)\phi\|_{L^\infty(\R^d)}\le \|\Pi_\lambda\|_{L^1\to
    L^\infty}\|\phi\|_{L^1(\R^d)}. 
\end{equation*}
We readily have $\|\Pi_\lambda\|_{L^1\to
    L^\infty}\le \|K_\lambda\|_{L^\infty_{x,y}}\lesssim \lambda^{d/2}$, where the
  last inequality stems from \eqref{eq:est-K-lambda}. This
  completes the proof of the lemma. 
\end{proof}

For any interval $I \subset [0,\infty)$, we define the space $\ell^q
(n \tau \in I;\, L^r (\R^d))$, or simply $\ell^q
( I;\, L^r )$,  as consisting of functions defined on $\tau \mathbb{Z}
\cap I$ with values in $L^r (\R^d)$, the norm of which is given by 
\begin{equation}\label{eq:ellqLr}
  \|u\|_{\ell^q (I;\, L^r )} = \left\{
    \begin{aligned}
      \Big( \tau \sum_{n\tau \in I} \|u(n \tau)\|_{L^r (\R^d)}^q \Big)^{1/q}
      & \qquad\text{ if }1\le q<\infty,\\
      \sup_{n\tau\in I} \|u(n \tau)\|_{L^r (\R^d)}& \qquad\text{ if } q=\infty.
    \end{aligned}
    \right.
\end{equation}
As $\lambda\ge 1$ (we will choose $\lambda\to \infty$ as $\tau\to
0$), Lemma~\ref{lem:discHLS} shows that in Lemma~\ref{lem:disp-disc},
the factor involving $\lambda$ is dominant in  the discrete Strichartz
estimate: 
proceeding like in \cite{ORS21} (a sketch of the argument is given
below), we infer 

\begin{proposition}\label{prop:strichartz-discret}
Let $(q,r)$,  $(a_1,b_1)$ and $(a_2, b_2)$ be  admissible
pairs. Then, there exist $C_{d,q}, C_{d,a_1,a_2}>0$ such that for any
bounded interval $I$ such that $|I|\le 1$,  provided that
$\lambda\tau\ge 1$,
    \begin{equation}\label{eq:discr-stri-homo}
    \|S_{\lambda}(\cdot) \phi \|_{\ell^q (I; L^r )}
    \le C_{d,q}  \(\lambda\tau\)^{1/q}\|\phi \|_{L^2 },
    \end{equation}
and
    \begin{equation}\label{eq:disc-stri-inhom}
    \left\| \tau \sum_{{k=-\infty}\atop{k\tau\in I}}^{n-1}S_{\lambda}\((n-k) \tau\) f
      (k\tau)\right\|_{\ell^{a_1} (I; L^{b_1} )} 
    \le C_{d,a_1,a_2}
    \(\lambda\tau\)^{\frac{1}{a_1}+\frac{1}{a_2}}\|f\|_{\ell^{a_2'}(I;
      L^{b_2'} )} 
    \end{equation}
hold for all $\phi\in L^2 $ and $f\in \ell^{a_2'}( I; L^{b_2'})$.
\end{proposition}
The proof essentially follows the same strategy as in \cite{KeelTao},
as adapted to the discrete in time case in \cite{Ignat2011}: since
$\|S_\lambda\|_{L^2\to L^2}\le 1$, Lemma~\ref{lem:disp-disc} and
interpolation yield, for $p\in [2,\infty]$,
 \begin{equation*}
    \|S_\lambda(t)\phi\|_{L^p(\R^d)}\le
      \frac{C}{\lambda^{-d(1/2-1/p)}+|t|^{d(1/2-1/p)}}\|\phi\|_{L^{p'}(\R^d)},\quad
      \forall \phi\in \Sch(\R^d),\ |t|\le \delta. 
\end{equation*}
By a $TT^*$ argument,  \eqref{eq:discr-stri-homo} follows from
\eqref{eq:disc-stri-inhom}. By duality, \eqref{eq:disc-stri-inhom} is
equivalent to 
\begin{equation*}
  \left| \tau^2\<\sum S_{\lambda}(k\tau )^* f
   (k\tau), \sum S_{\lambda}(n\tau )^* g
   (n\tau)\>\right|
  \le C \(\lambda \tau\)^{2/a_2}
   \|f\|_{\ell^{a_2'}
    L^{b_2'} } \|f\|_{\ell^{a_2'}
    L^{b_2'} } ,
\end{equation*}
where we omit the summation set to lighten notations. Using H\"older
inequality, 
\begin{align*}
  \left|\<S_{\lambda}(k\tau )^* f (k\tau), S_{\lambda}(n\tau )^* g
  (n\tau)\>\right|
  & = \left|\< f (k\tau), S_{\lambda}((k-n)\tau ) g
    (n\tau)\>\right|\\
  \le \|f(k\tau)\|_{L^{b_2'}} &\left\|  S_{\lambda}((k-n)\tau ) g
    (n\tau)\right\|_{L^{b_2}}\\
   \le \|f(k\tau)\|_{L^{b_2'}}&
    \frac{C}{\lambda^{-d(1/2-1/b_2)}+|((k-n)\tau|^{d(1/2-1/b_2)}}
    \|g(n\tau)\|_{L^{b_2'}}. 
\end{align*}
Since $(a_2,b_2)$ is admissible, $d(1/2-1/b_2) = 2/a_2$, so H\"older
inequality in time now yields
\begin{align*}
  \Big| \tau^2\Big\langle \sum S_{\lambda}(k\tau )^* f
  (k\tau), \sum S_{\lambda}(n\tau )^*
  &g
   (n\tau)\Big\rangle\Big|\\
 \le C  \|f\|_{\ell^{a_2'} L^{b_2'}}&\left\|  \sum_n
   \frac{\|g(n\tau)\|_{L^{b_2'}}}{(\lambda\tau)^{-2/a_2}+|k-n|^{2/a_2}}
 \right\|_{\ell^{a_2}}. 
\end{align*}
We  estimate the last term by invoking
Lemma~\ref{lem:discHLS} with $\alpha=2/a_2$, $q= a_2$ and $p=a_2'$.
Similarly, the analogue of \cite[Lemma~4.5]{Ignat2011} (see also
\cite[Corollary~2.4]{ChoiKoh2021}) is:
\begin{corollary}\label{cor:Stri-cont-disc}
  Let $(a_1,b_1)$ and $(a_2,b_2)$ be  admissible pairs, and
  $\lambda=1/\tau$. There exists $C=C(d,a_1,a_2)$ such that  for any
bounded interval $I$ such that $ |I|\le 1$,
\begin{equation*}
  \left\| \int_{{s<n\tau}\atop{s\in I}}S_\lambda
    (n\tau-s)f(s)ds\right\|_{\ell^{a_1}(I;L^{b_1})} \le C \|f\|_{L^{a_2'}(I;
      L^{b_2'} )} ,\quad \forall f\in L^{a_2'}( I; L^{b_2'}).
\end{equation*}
\end{corollary}

\section{Cauchy problem}
\label{sec:cauchy}

\subsection{Well-posedness in $\cH^1$}
We adapt  well-posedness results known in $\Sigma$ for
\eqref{eq:NLSP}, under Assumptions~\ref{hyp:V}.
For a time interval $I$, denote
\begin{equation}\label{eq:X1}
  \|u\|_{X^1(I)} = \max_{A\in \{\Id,\nabla, \sqrt V \}}\sup_{(q,r)\text{
      admissible}}\|Au \|_{L^q(I;L^r)},
\end{equation}
and let $X^1(I)$ be the corresponding (Banach) space. Note that
$X^1(I) \subset  C(I;\cH^1)$. When $I=[0,T]$, we
simply denote $X^1([0,T])$ by $X^1_T$. 
\begin{proposition}\label{prop:cauchy1}
  Let Assumption~\ref{hyp:V} be verified and
  $0<\si<\frac{2}{(d-2)_+}$. For any 
  $u_0\in \cH^1$, there exist $T>0$  a unique solution $u\in X^1_T$ to
  \eqref{eq:NLSP}.
 In addition, its mass and energy are independent
of time,
\begin{align*}
  &\frac{d}{dt}\int_{\R^d} |u(t,x)|^2dx=0,\\
& \frac{d}{dt}\int_{\R^d}\(|\nabla u(t,x)|^2 
   +V(x)|u(t,x)|^2 +\vp(x)|u(t,x)|^2 + \frac{\eps}{\si+1}|u(t,x)|^{2\si+2}\)dx=0. 
\end{align*}
Finally, the above statement is true for all $T>0$ (global solution)
if $0<\si<2/d$ ($L^2$-subcritical case), or $\eps=+1$ (defocusing
case). 
\end{proposition}
\begin{proof}[Main steps of the proof]
Denote by 
\begin{equation*}
  (q_0,r_0)= \( \frac{4\si+4}{d\si},2\si+2\),
\end{equation*}
the pair present in the above statement, which is admissible (see
Definition~\ref{def:adm}). 

Local existence follows from a fixed point argument on Duhamel's
formula associated to \eqref{eq:NLSP},
\begin{equation}\label{eq:duhamel}
   u(t) =S(t)u_0 -i\int_0^t
   S(t-s)\(\vp u\)(s)ds-i\eps\int_0^t
   S(t-s)\(|u|^{2\si}u\)(s)ds, 
\end{equation}
in a ball of the space
\begin{equation*}
  Z_T=\{u\in C([0,T];\cH^1);\ u,\nabla u,u\sqrt V \in
  L^{q_0}(0,T;L^{r_0}(\R^d)) \}, 
\end{equation*}
which, equipped with the distance
\begin{equation*}
d(u,v)=\|u-v\|_{L^{q_0}_TL^{r_0}}+\|u-v\|_{L^{\infty}_TL^{2}}, 
\end{equation*}
is a complete Banach space (see e.g. \cite[Section~4.4]{CazCourant}).  Let
\begin{equation}\label{eq:theta}
   \theta=\frac{2\si(2\si+2)}{2-(d-2)\si}. 
\end{equation}
Then we have
\begin{equation*}
   \frac{1}{q_0'}=\frac{2\si}{\theta} +\frac{1}{q_0} \quad ;\quad
\frac{1}{r_0'}=\frac{2\si}{r_0}+\frac{1}{r_0}.
\end{equation*}
Denoting by $\Phi(u)(t)$ the right hand side of \eqref{eq:duhamel},  Strichartz estimates (Proposition~\ref{prop:strichartz}) yield,
for $(q,r)$ admissible,
\begin{equation}\label{eq:holder}
\begin{aligned}
  \|\Phi(u)\|_{L^q_TL^r}&\lesssim \|u_0\|_{L^2}+ \|\vp u\|_{L^1_TL^2}+\left\|
    |u|^{2\si}u\right\|_{L^{q_0'}_TL^{r_0'}}\\
& \lesssim
  \|u_0\|_{L^2}+ T\|\vp\|_{L^\infty}\|u\|_{L^\infty_TL^2}
+\|u\|_{L^\theta_TL^{r_0}}^{2\si}  \|u\|_{L^{q_0}_TL^{r_0}},  
\end{aligned}
\end{equation}
where we have used H\"older inequality. In view of Sobolev
embedding, 
\begin{equation}\label{eq:Phi1}
  \|\Phi(u)\|_{L^q_TL^r} \le C \|u_0\|_{L^2}
  + T\|\vp\|_{L^\infty}\|u\|_{L^\infty_TL^2}
+C T^{2\si/\theta}\|u\|_{L^\infty_TH^1}^{2\si}  \|u\|_{L^{q_0}_TL^{r_0}}. 
\end{equation}

Unlike in the case without potential, $V=0$, the group $S(t)$ and the
operator $\nabla$ do not commute: we may either consider the
commutators $[\nabla,H]= \nabla V$ and $[\sqrt V,H] = -\nabla \sqrt
V\cdot \nabla - \frac{1}{2}\Delta\sqrt V$, and get a closed system of
estimates (see e.g. \cite{Ca15}), or use directly the operator
$\sqrt H$, and rely on Lemma~\ref{lem:sobolev-H}, which implies 
\begin{equation*}
  Z_T=\{u\in C([0,T];\cH^1);\ \sqrt H u \in
  L^{q_0}(0,T;L^{r_0}(\R^d)) \}.
\end{equation*}
We choose the
latter option, as it paves the way to the analysis of error estimates
in the discrete in time setting. We have
\begin{align*}
  \sqrt H \Phi(u)(t) &= \sqrt HS(t) u_0 -i\sqrt H\int_0^t
  S(t-s) \(\vp u\)(s)ds\\
&\quad-i\eps \sqrt H\int_0^t
  S(t-s) \(|u|^{2\si}u\)(s)ds\\
&=S(t)  \sqrt Hu_0-i\int_0^t
  S(t-s) \sqrt H\(\vp u\)(s)ds\\
&\quad -i\eps\int_0^t S(t-s ) \sqrt H\(|u|^{2\si}u\)(s)ds.
\end{align*}
We resume the same Lebesgue exponents as above when applying
Strichartz estimates:
\begin{equation*}
  \|\sqrt H\Phi(u)\|_{L^q_TL^r} \lesssim \|\sqrt Hu_0\|_{L^2} 
+\|\sqrt H(\vp u)\|_{L^1_TL^2}
+  \left\| \sqrt H\(|u|^{2\si}u\)\right\|_{L^{q_0'}_TL^{r_0'}}.
\end{equation*}
In view of \eqref{eq:H12der} and H\"older inequality like in
\eqref{eq:holder}, 
\begin{equation*}
  \left\| \sqrt H\(|u|^{2\si}u\)\right\|_{L^{q_0'}_TL^{r_0'}}\lesssim
  \|u\|_{L^\theta_TL^{r_0}}^{2\si }\|\sqrt H   u\|_{L^{q_0}_TL^{r_0} }.
\end{equation*}
Using Sobolev embedding, Lemma~\ref{lem:sobolev-H}, and \eqref{eq:H12der}  again,
\begin{align*}
   \|\sqrt H\Phi(u)\|_{L^q_TL^r} &\lesssim \|\sqrt Hu_0\|_{L^2}
                                   +T\|\sqrt H u\|_{L^\infty_T L^2}+ 
   T^{1/\theta}\|u\|_{L^\infty_T H^1}^{2\si }\|\sqrt H   u\|_{L^{q_0}_TL^{r_0} }\\
&\lesssim \|\sqrt Hu_0\|_{L^2} +T\|\sqrt H u\|_{L^\infty_T L^2}+
   T^{1/\theta}\|\sqrt H u\|_{L^\infty_T L^2}^{2\si }\|\sqrt H   u\|_{L^{q_0}_TL^{r_0} }.
\end{align*}
Choosing successively $(q,r)=(\infty,2)$ and $(q_0,r_0)$, we have a
closed system of inequalities, and so picking $T=T(\|u_0\|_{\cH^1})$
sufficiently small, we can prove that $\Phi$ maps 
a suitable ball in $Z_T$ into itself. Contraction for the norm
$\|\cdot\|_{L^{q_0}_TL^{r_0}}+\|\cdot\|_{L^\infty_TL^2} $ is proved similarly, and local
existence follows.
\smallbreak

Uniqueness stems from the same estimates as
above. We 
refer to \cite{CazCourant} for the rigorous justification of the
conservation of mass and energy. Globalization is a
consequence of these
conservations, and Gagliardo--Nirenberg inequality when
$\eps=-1$ and $\si<2/d$. 
\end{proof}

\subsection{Well-posedness in $\cH^2$}
\label{sec:cauchyH2}
For $T>0$, denote
\begin{equation}\label{eq:X2}
  \|u\|_{X^2(I)} = \max_{A\in \{\Id,\Delta,  V \}}\sup_{(q,r)\text{
      admissible}}\|Au \|_{L^q(I;L^r)},
\end{equation}
and let $X^2(I)$ be the corresponding (Banach) space. Note that
$X^2(I) \subset  C(I;\cH^2)$. When $I=[0,T]$, we
simply denote $X^2([0,T])$ by $X^2_T$.
\begin{proposition}\label{prop:cauchy2}
  Let Assumption~\ref{hyp:V} be verified and $0<\si<\frac{2}{(d-2)_+}$. For any
  $u_0\in \cH^2$, there exist $T>0$  a unique solution $u\in X^1_T\cap
  C([0,T];\cH^2)$ to
  \eqref{eq:NLSP}. If in addition $\si\ge 1/2$, $u\in X^2_T$. The time
  $T>0$ can be chosen arbitrarily large (global solution) 
if $0<\si<2/d$ ($L^2$-subcritical case), or $\eps=+1$ (defocusing
case). 
\end{proposition}
\begin{proof}
  Since $\cH^2\subset \cH^1$, Proposition~\ref{prop:cauchy1} provides
  a local solution in $X^1_T$ for some $T>0$, which is arbitrarily
  large in the cases addressed at the end of the statement of the
  proposition. Arguing as in \cite[Section~5.3]{CazCourant}, we note that applying
  the operator $\partial_t$ in \eqref{eq:NLSP}, we find that $\partial_t u$ solves
  \begin{equation*}
    \(i\partial_t -H\)\partial_t u =\vp \partial_t u+ \eps \partial_t \(|u|^{2\si}u\), \quad \partial_t u_{\mid
      t=0} = -iHu_0-i\vp u_0 -i\eps|u_0|^{2\si}u_0.
  \end{equation*}
We check that since $u_0\in \cH^2\subset H^2(\R^d)$,
$|u_0|^{2\si}u_0\in L^2(\R^d)$ by Sobolev embedding $H^2\subset
H^{d\si/(2\si+1)}\subset 
L^{4\si+2}$. We also note that in view of Lemma~\ref{lem:sobolev-H},
\begin{equation*}
\cH^2 = \{\phi\in L^2(\R^d),\quad H\phi\in L^2(\R^d)\},  
\end{equation*}
and, in the sense of the equivalence of norms,
\begin{equation*}
  \|u\|_{X^2(I)}\approx \sup_{(q,r)\text{
      admissible}}\|Hu \|_{L^q(I;L^r)}.
\end{equation*}
Writing Duhamel's formula for $\partial_t u$ and resuming the estimates from
the proof of Proposition~\ref{prop:cauchy1}, we see that $\partial_t u \in
L^q_TL^r$ for all admissible pairs $(q,r)$. In particular, $\partial_t u \in
C([0,T];L^2)$ (see e.g. \cite{CazCourant} for continuity). At this
stage, we know:
\begin{equation*}
  \underbrace{i\partial_tu}_{\in C([0,T];L^2)} = Hu + \underbrace{\vp
    u}_{\in  C([0,T];L^2)} + \eps|u|^{2\si}u.
\end{equation*}
We show that the linear term $Hu$ controls the nonlinear
term, in $L^2$:
\begin{equation*}
 \|\phi\|_{L^{4\si+2}}^{2\si+1} \lesssim 
\|\phi\|_{H^{d\si/(2\si+1)}}^{2\si+1} 
\lesssim \|\phi\|^{2\si}_{H^1}\|\phi\|_{H^s}
  ,
\end{equation*}
with $s= 1$ if $d\si\le 2\si+1$, and $s=(d-2)\si<2$ if $d\si>
2\si+1$, so there is always $\eta\in (0,1)$ such that
\begin{equation*}
 \|\phi\|_{L^{4\si+2}}^{2\si+1} \lesssim
 \|\phi\|^{2\si+1-\eta}_{H^1}\|\phi\|^{\eta}_{H^2}  \lesssim
 \|\sqrt H\phi\|^{2\si+1-\eta}_{L^2}\|H\phi\|^{\eta}_{L^2} .
\end{equation*}
In view of Young inequality, since $u\in X_T^1$, there exists $C$ such that for all $t\in [0,T]$,
\begin{equation*}
  \|u(t)\|_{L^{4\si+2}}^{2\si+1}\le C + \frac{1}{4} \| H u(t)\|_{L^2}.
\end{equation*}
We infer that $Hu\in C([0,T];L^2)$, hence $u\in C([0,T];\cH^2)$. If in
addition $\si\ge 1/2$, we can differentiate the nonlinearity twice, 
and so
\begin{equation*}
  \(i\partial_t -H\)H u =H(\vp u)+ \eps H\(|u|^{2\si}u\)= \Delta(\vp u) +
  V\vp u+\eps\Delta
  \(|u|^{2\si}u\)+\eps V |u|^{2\si}u. 
\end{equation*}
We can then prove that $u\in X^2_T$. We emphasize that unlike what is claimed in
\cite{ChoiKoh2021,Ignat2011}, even in the case 
$V=0$, the proof that $\Delta u\in L^q_TL^r$ for admissible pairs with
$r\not =2$ is not stated in \cite[Section~5.3]{CazCourant}. To
simplify the presentation, we assume $\vp=0$, as the corresponding
term bears no extra difficulty. We
distinguish two cases:\\
$\bullet $ If $0<\si<2/d$, we have $1/q_0<1/\theta$, and thus
\begin{equation*}
  \|u\|_{L^\theta_TL^{r_0}}\le T^{1/\theta-1/q_0} \|u\|_{L^{q_0}_TL^{r_0}}.
\end{equation*}
Strichartz
estimates like in \eqref{eq:holder}
Lemma~\ref{lem:sobolev-H}  yield, on $I=[t_j,t_{j+1}]\subset [0,T]$,
\begin{align*}
  \|H u\|_{L^q(I;L^r)}& \lesssim \|H u(t_j)\|_{L^2} +\|u\|_{L^1(I;\cH^2)}+
\|u\|_{L^\theta(I;L^{r_0})}^{2\si} \|V u\|_{L^{q_0}(I;L^{r_0})}\\
 +\|u\|_{L^\theta(I;L^{r_0})}^{2\si-1} &\(\|\nabla
  u\|_{L^\theta(I;L^{r_0})} \|\nabla  u\|_{L^{q_0}(I;L^{r_0})}+\|
u\|_{L^\theta(I;L^{r_0})} \|\Delta  u\|_{L^{q_0}(I;L^{r_0})}\)\\
  & \lesssim \|H u(t_j)\|_{L^2} +|I|
    \|u\|_{L^\infty(I;\cH^2)}+\|u\|_{L^{q_0}(I;L^{r_0})}^{2\si} 
    \|H u\|_{L^{q_0}(I;L^{r_0})}\\
&\quad  +\|u\|_{L^{q_0}(I;L^{r_0})}^{2\si-1} \|\sqrt H  u\|_{L^{q_0}(I;L^{r_0})}^2.
\end{align*}
If $u\in X_T^1$, then in particular $u \in
L^{q_0}([0,T];L^{r_0})$, so we can write $[0,T]$ as the union of
finitely many intervals on which $\|u\|_{L^{q_0}(I;L^{r_0})}$ is
sufficiently small so the nonlinear  terms are absorbed by the left hand side
when choosing $(q,r)=(q_0,r_0)$,
up to doubling the constants on each intervals
$[t_j,t_{j+1}]$. We infer $ Hu \in L^{q_0}_TL^{r_0}$, and resuming the
above estimate with $(q,r)$ an arbitrary admissible pair, we conclude
that $u\in X_T^2$. \\
\smallbreak

\noindent $\bullet$ If $2/d\le \si<2/(d-2)_+$, we resume the argument from
\cite[Section~2.2]{CaSu24}, and change the above estimates to 
\begin{align*}
  \|H u\|_{L^q(I;L^r)}& \lesssim \|H u(t_j)\|_{L^2} 
 +\|Hu\|_{L^\infty_TL^2}+
\|u\|_{L^\theta(I;L^{r_0})}^{2\si} \|V u\|_{L^{q_0}(I;L^{r_0})} \\
 +\|u\|_{L^\theta(I;L^{r_0})}^{2\si-1} &\(\|\nabla
  u\|_{L^\theta(I;L^{r_0})} \|\nabla  u\|_{L^{q_0}(I;L^{r_0})}+\|
u\|_{L^\theta(I;L^{r_0})} \|\Delta  u\|_{L^{q_0}(I;L^{r_0})}\)\\
  & \lesssim \|H u(t_j)\|_{L^2} +\|Hu\|_{L^\infty_TL^2}+
 \|u\|_{L^\theta(I;L^{r_0})}^{2\si} \|V u\|_{L^{q_0}(I;L^{r_0})}  \\
 +\|u\|_{L^\theta(I;L^{r_0})}^{2\si-1} &\(\|
  u\|_{L^\theta(I;W^{2,b_1})} \|\nabla  u\|_{L^{q_0}(I;L^{r_0})}+\|
u\|_{L^\theta(I;L^{r_0})} \|\Delta  u\|_{L^{q_0}(I;L^{r_0})}\),
\end{align*}
where $b_1\ge 2$ is such that $(\theta,b_1)$ is admissible: indeed,
$\theta>2$ when $\si\ge 2/d$, and moreover
$W^{s,b_1}(\R^d)\hookrightarrow L^{r_0}(\R^d)$ with
\begin{equation*}
  s = d\(\frac{1}{b_1}-\frac{1}{r_0}\) =
  \frac{2}{q_0}-\frac{2}{\theta}=  \frac{d\sigma-2}{2\si}\in
  [0,1). 
\end{equation*}
If $u\in X_T^1$, then in particular $u,\nabla u \in
L^{q_0}([0,T];L^{r_0})$, and we can follow essentially the same lines as in the
case $\si<2/d$. 
\end{proof}

The following corollary explains why, in the statements of
Theorems~\ref{theo:main} and \ref{theo:main2}, we do not assume that
$u\in X_T^1$ or $u\in X_T^2$, even though the proof of these results
will rely on such properties:
\begin{corollary}\label{cor:reg}
  Let $u_0\in \cH^1$. Either the solution provided by
  Proposition~\ref{prop:cauchy1} is global, in the sense that $T>0$ is
  arbitrary, or there exists $T^*>0$ such that 
  \begin{equation*}
    \|\nabla u(t)\|_{L^2(\R^d)}\Tend t {T^*}+\infty. 
  \end{equation*}
In particular, if $u\in C([0,T];\cH^1)$, then $u\in X_T^1$. \\
If in addition $u_0\in \cH^2$ and $\si\ge 1/2$, a similar statement
holds: if $u \in C([0,T];\cH^1)$, then $u\in X_T^2$. 
\end{corollary}
\begin{proof}
  The construction of a solution in the proof of
  Proposition~\ref{prop:cauchy1} relies on a fixed point argument,
  which provides a local existence time
  $T=T(\|u_0\|_{\cH^1})$. Standard ODE arguments imply that either the
  solution is global, $u\in X_T^1$ for all $T>0$, or its $\cH^1$-norm
  becomes infinite in finite time,
  \begin{equation*}
    \exists T^*>0,\quad \|u(t)\|_{\cH^1}\Tend t {T^*}+\infty. 
  \end{equation*}
Recall that the $L^2$-norm of $u$ is independent of time. We now examine
the conservation of the energy. The conservation of mass implies  $\vp
u\in L^\infty([0,T^*];L^2)$. If $\nabla 
u\in L^\infty([0,T^*];L^2)$, then the Sobolev embedding
$H^1(\R^d)\hookrightarrow L^{2\si+2}(\R^d)$  implies that $u\in
L^\infty([0,T^*];L^{2\si+2})$. In the conserved energy,
three terms out of four are bounded, so we infer $u\sqrt V \in
L^\infty([0,T^*];L^2)$, and thus $u\in 
L^\infty([0,T^*];\cH^1)$: therefore, either the solution is global, or
\begin{equation*}
   \exists T^*>0,\quad \|\nabla u(t)\|_{L^2}\Tend t {T^*}+\infty. 
 \end{equation*}
 In view of \cite[Proposition~4.2.1]{CazCourant}, uniqueness holds for
 solutions of \eqref{eq:NLSP} which belong to
 $C([0,T];\cH^1)$, hence the first part of the corollary.
In the proof of Proposition~\ref{prop:cauchy2}, we have seen  that if $u\in
X_T^1$, then $u\in X_T^2$, hence the corollary.   
\end{proof}

\section{Stability implies convergence}
\label{sec:cv}

The goal of this section is to establish the convergence
results of Theorems~\ref{theo:main} and \ref{theo:main2}, when stability is assumed:
\begin{theorem}\label{theo:stab-cv}
   Let Assumption~\ref{hyp:V} be verified. 
Suppose that for some 
   $T>0$, 
   \eqref{eq:NLSP} has a unique solution  $u\in 
   X^1_T$. Denote by $u^n$ the sequence 
defined by the scheme \eqref{eq:modified-splitting}, where we set
$\lambda=1/\tau$. Suppose that there exist $\gamma,M>0$ such
that for all interval $I\subset [0,T]$, the
numerical solution satisfies
\begin{equation}\label{eq:stab-num}
  \|u^n\|_{\ell^\theta(I;L^{r_0})}
% +  \|x u^n\|_{\ell^{q_0}([0,T];L^{r_0})} 
\le |I|^\gamma M,\quad \forall
  \tau\in (0,1),\quad \theta=\frac{2\si(2\si+2)}{2-(d-2)\si},\ r_0=2\si+2,
\end{equation}
where the above $\ell^\theta$ norm is defined in
\eqref{eq:ellqLr}. Let $(q,r)$ be an admissible pair.
%, and the values of $\theta$ and $r_0$ were introduced in
%Section~\ref{sec:cauchy}. 
\\
$\bullet$ Assume that either $0<\si<2/d$, or $d\le 5$ and $2/d\le
\si<2/(d-2)_+$, with in addition $\si\ge 1/2$ when $d=5$. 
There exists $C$ such that
\begin{equation*}
  \|u^n-u(n\tau)\|_{\ell^q([0,T];L^r)} \le C\tau^{1/2}.
\end{equation*}
$\bullet$ If $1/2\le \si<2/(d-2)_+$, and $u\in  X^2_T$, 
then there exists $C$ such that
\begin{equation*}
  \|u^n-u(n\tau)\|_{\ell^q([0,T];L^r)} \le C\tau.
\end{equation*}
$\bullet$ If $1/2\le \si<2/(d-2)_+$, $u\in  X^2_T$, and, up to
increasing $M$, 
\begin{equation}\label{eq:stab-num2}
  \|u^n\|_{\ell^\infty(I;L^{r_0})}
\le M,\quad \forall
  \tau\in (0,1),
\end{equation}
then there exists $C$ such that
\begin{equation*}
 \left\|H^{1/2}\(u^n-u(n\tau)\)\right\|_{\ell^q([0,T];L^r)}\le C\tau^{1/2}. 
\end{equation*}
\end{theorem}
Setting $(q,r)=(\infty,2)$, we get the conclusions of
Theorems~\ref{theo:main} and \ref{theo:main2}.

We note that the condition \eqref{eq:stab-num2} is stronger than
\eqref{eq:stab-num}, since \eqref{eq:stab-num2} and H\"older
inequality imply
\begin{equation*}
  \|u^n\|_{\ell^\theta(I;L^{r_0})}
\le |I|^{1/\theta}\|u^n\|_{\ell^\infty(I;L^{r_0})}
\le |I|^{1/\theta} M .
\end{equation*}
The proof of this result is very similar to the proof of
\cite[Theorem~1.2]{Ignat2011} and
\cite[Theorem~1.4]{ChoiKoh2021}. We note however that the assumption
\eqref{eq:stab-num} is weaker than the one made in
\cite[Theorem~1.4]{ChoiKoh2021}: it turns out that in the proof of
\cite[Theorem~1.4]{ChoiKoh2021}, it is precisely \eqref{eq:stab-num}
which is used. 
We recall the main steps of the arguments, with enough details so it
should be clear that only the condition \eqref{eq:stab-num} is required
on the numerical solution, and emphasize the modifications due to the
presence of the potential $V$ here.

\subsection{Preliminary lemmas}

We first note that 
\eqref{eq:Pi-lambda-Lp}, \eqref{eq:Pi-lambda-1}, 
Lemmas~\ref{lem:perte} and \ref{lem:bernstein} provide the
analogue of \cite[Lemma~2.6]{ChoiKoh2021}. Also, contrary to what
happens in \cite{ChoiKoh2021,Ignat2011}, the operators $\Pi_\lambda$
and $\nabla$ do not commute: this is why we consider the operator
$H^{1/2}$ instead of $\nabla$, using \eqref{eq:H12der}. 
 The next result corresponds essentially to
\cite[Lemma~2.5]{ChoiKoh2021}, except for the last statement, which
appears in the proof of \cite[Lemma~4.3]{Ignat2011}, and is easily
deduced from explicit computations and \eqref{eq:NLpuissance}: 

\begin{lemma}\label{lem:propNL}
Denote by $N_0$ the map $N$ when $\vp=0$,
$  N_0(t)\phi = \phi^{-i\eps t|\phi|^{2\si}}$.
There exists  $c>0$ such that
    \begin{equation}\label{eq:NLlip}
    \left| \frac{N_0(\tau) - \Id}{\tau} v - \frac{N_0(\tau) - \Id}{\tau}w \right|
    \le c\(|v|^{2\si} + |w|^{2\si}\)|v-w|
    \end{equation}
and
    \begin{equation}\label{eq:NLpuissance}
    \left| \frac{N_0(\tau)-\Id}{\tau} v\right|
    = \left| \frac{\exp(-i \tau \eps|v|^{2\si}) -1}{\tau}
      v\right| \le |v|^{2\si+1} 
    \end{equation}
hold for all $v, w \in \mathbb{C}$.
Furthermore, for $f$ smooth enough, we have the pointwise estimates
    \begin{equation}\label{eq:gradientNL}
    \left| \nabla \left( \frac{N_0(\tau) - \Id}{\tau} f \right) \right|
    \lesssim  |f|^{2\si} |\nabla f|,
  \end{equation}
and, if $\si\ge 1/2$ and $0<\tau\le 1$,
   \begin{equation}\label{eq:laplaceNL}
    \left| \Delta \left( \frac{N_0(\tau) - \Id}{\tau} f \right) \right|
    \lesssim  |f|^{2\si} |\Delta f|+ |f|^{2\si-1}|\nabla f|^2 + \tau
    |f|^{4\si-1} |\nabla f|^2. 
  \end{equation}
\end{lemma}
Together with Lemma~\ref{lem:sobolev-H} and H\"older inequality, the
above lemma entails: 
\begin{lemma}\label{lem:H12N0}
  Let $1< p<\infty$. The exist constants such that for all $f\in
  \Sch(\R^d)$,
  \begin{equation}
    \label{eq:gradientNLLp}
    \left\| \sqrt H \(\frac{N_0(\tau) - \Id}{\tau} f \) \right\|_{L^p}
    \lesssim \|f\|^{2\si}_{L^{a_1}}\|\sqrt H f\|_{L^{a_2}},\quad
    \frac{1}{p}= \frac{2\si}{a_1}+\frac{1}{a_2},
  \end{equation}
  and, if $\si\ge 1/2$ and $0<\tau\le 1$,
  \begin{equation}
    \label{eq:laplaceNLLp}
    \begin{aligned}
   \left\| H\( \frac{N_0(\tau) - \Id}{\tau} f \) \right\|_{L^p}
  &  \lesssim \|f\|^{2\si}_{L^{a_1}}\| H f\|_{L^{a_2}} +
  \|f\|^{2\si-1}_{L^{a_3}}\| \sqrt H f\|_{L^{a_3}}^2 \\
  &\quad+ \tau
    \|f\|_{L^{a_5}}^{4\si-1} \|\sqrt H f\|_{L^{a_6}}^2,\\
    \frac{1}{p} = \frac{2\si}{a_1}+\frac{1}{a_2} &=
    \frac{2\si-1}{a_3} + \frac{2}{a_4} =
    \frac{4\si-1}{a_5}+\frac{2}{a_6},
    \end{aligned}
  \end{equation}
with, for all $j$'s, $a_j\in [1,\infty]$.
\end{lemma}

The analogue of \cite[Lemma~4.6]{Ignat2011} and
\cite[Lemma~2.7]{ChoiKoh2021} 
is rather straightforward: 
\begin{lemma}\label{lem-5-2}
For any admissible pairs $(a_1,b_1)$ and $(a_2,b_2)$, any interval
$I=[0,T]$ of length $T\le 1$, and $\beta\in\{1,2\}$, there is a
constant $C=C_{d,a_1,a_2}>0$ such that if $\lambda=1/\tau$,  for
any test function $f \in \mathcal{S}(\R^{d+1})$, 
   \begin{equation}\label{eq-5-5bis}
    \begin{aligned}
    &\left\| \int_0^{n\tau}S_{\lambda}(n \tau- s) f(s) ds
        - \tau \sum_{k=0}^{n-1} S_{\lambda}(n \tau - k \tau) f(k
        \tau) \right\|_{\ell^{a_1} (I; L^{b_1} )} \\
    &\qquad\le C \tau^{\beta/2} \|H^{\beta/2} f\|_{L^{a_2'}(I;
      L^{b_2'})}
    + C~ \tau \| \partial_t f \|_{L^{a_2'}(I; L^{b_2'})}.
    \end{aligned}
    \end{equation}
\end{lemma}
\begin{proof}[Main steps of the proof]
  In the case $V=0$, this result appears in
  \cite[Lemma~4.6]{Ignat2011} when $\beta=2$, and in
  \cite[Lemma~2.7]{ChoiKoh2021} when $\beta=1$.  
 Following the proof of \cite[Lemma~4.6]{Ignat2011}, the quantity to
 estimate is rewritten as
 \begin{align*}
   \sum_{k=0}^{n-1} & \int_{k\tau}^{(k+1)\tau} \(S_\lambda
   (n\tau-s)f(s)-S_\lambda(n\tau-k\tau)f(k\tau)\)ds\\
   &= \sum_{k=0}^{n-1}\int_{k\tau}^{(k+1)\tau}
     \int_{k\tau}^s\frac{d}{dt}\( S_\lambda(n\tau-t)f(t)\)dtds\\
   & =\sum_{k=0}^{n-1} \iint_{k\tau< t<s<(k+1)\tau}
     \(iS_\lambda(n\tau-t)Hf(t) + S_\lambda(n\tau-t)\partial_t f(t)\)dtds\\
   & =\sum_{k=0}^{n-1} \int_{k\tau}^{(k+1)\tau}
   \((k+1)\tau-t\)  \(iS_\lambda(n\tau-t)Hf(t) +
     S_\lambda(n\tau-t)\partial_t f
     (t)\)dt \\
    & =\sum_{k=0}^{n-1} \int_{k\tau}^{(k+1)\tau}
  S_\lambda(n\tau-t) \((k+1)\tau-t\)  \(iHf(t) + \partial_t f(t)\)dt.
 \end{align*}
 Then \eqref{eq-5-5bis} follows from the inhomogeneous discrete
 Strichartz estimate from Corollary~\ref{cor:Stri-cont-disc}, applied to
 \begin{equation*}
  g(t)=  \sum_0^{n-1}\( (k+1)\tau-t\) \(iHf(t)+\partial_tf(t)\)
  \Id_{(k\tau,(k+1)\tau)}(t), 
\end{equation*}
and from triangle inequality. Passing from $\beta=2$ to
$\beta=1$ is a direct application of Lemma~\ref{lem:perte}.
\end{proof}

\subsection{Stability implies convergence in $L^2$}
\label{sec:cv-L2}

In this section, we prove the first two points of
Theorem~\ref{theo:stab-cv}. The argument is similar to the one
introduced in \cite{Ignat2011} and extended in \cite{ChoiKoh2021}, so
we describe the main steps,  and emphasize the main differences due to
the present framework.

Denote $I=[0,T]$. 
We estimate $\cZ_{\lambda}(n \tau)u_0 - \Pi_{\lambda} u(n \tau)$ instead of
$\cZ_{\lambda}(n \tau)u_0 - u(n \tau)$. Indeed,
Lemma~\ref{lem:PiPDO}  yields, for any admissible pair $(q,r)$,
since $\lambda=1/\tau$,
\begin{equation*}
\left\| u(n \tau) - \Pi_{\lambda} u(n \tau) \right\|_{\ell^q (I; L^r)}
\lesssim\frac{1}{\lambda^{\beta/2}}\| H^{\beta/2} u(n \tau)\|_{\ell^q (I; L^r)}
\lesssim\tau^{\beta/2},
\end{equation*}
where we assume $u\in X_T^\beta$, with $\beta\in \{1,2\}$. Indeed,
Corollary~\ref{cor:Stri-cont-disc} implies
\begin{equation*}
  \| H^{\beta/2} u(n \tau)\|_{\ell^q (I; L^r)} \lesssim \|
  H^{\beta/2}u_0\|+ \|H^{\beta/2}\vp u\|_{L^1_T L^2} + \left\|
 H^{\beta/2}  \( |u|^{2\si}u\)\right\|_{L^{q_0'}_TL^{r_0'}} ,
\end{equation*}
and the right hand side is bounded if $u\in X_T^\beta$, as we have
seen in the proofs of Proposition~\ref{prop:cauchy1} and
Proposition~\ref{prop:cauchy2}. 
To lighten notations, we now write $Z_{\tau}(n \tau)$ for
$\cZ_{\lambda}(n \tau)u_0=u^n$ with $\lambda=1/\tau$. 
In view of \eqref{eq:ellqLr} and Proposition~\ref{prop:cauchy1},
\begin{equation*}
  \|u(k\tau)\|_{\ell^{\theta}( I;L^{r_0})}
  \lesssim |I|^{1/\theta}\|u\|_{L^\infty_TH^1}\lesssim |I|^{1/\theta}\|u\|_{X_T^1}.
\end{equation*}
Keeping 
\eqref{eq:stab-num} into account, we infer that 
for any $\eta>0$, we can find a finite number
$K=K(\eta)$, and $\rho=\rho(\eta)>0$ with $\rho/\tau\in\N$ such that, if we set
\begin{equation*}
  I_j = [j\rho,(j+1)\rho]=:[m_j\tau,m_{j+1}\tau],\quad 0\le j\le K-1,
\end{equation*}
we have
\begin{equation*}
  [0,T]= \bigcup_{j=0}^{K-1}I_j \cup [K\rho,T]= \bigcup_{j=0}^{K}I_j,
\end{equation*}
and 
%$\tau(\eta)>0$ such that 
for $0\le j\le K$, 
%and $0<\tau\le \tau(\eta)$,
\begin{equation*}
   \|u(k\tau)\|_{\ell^{\theta}( I_j;L^{r_0})}+
   \|Z_\tau(k\tau)\|_{\ell^{\theta}(
    I_j;L^{r_0})}\le \eta.
\end{equation*}
On each interval $I_j$, the discrete Duhamel's formula can be written as
    \begin{equation}\label{eq-a-1'}
    Z_\tau(m_j \tau +n\tau)
    =S_{\lambda}(n\tau) Z_\tau(m_j\tau) +
    \tau \sum_{k=0}^{n-1} S_{\lambda}(n \tau -k \tau) \frac{N (\tau) -\Id}{\tau} Z_\tau(m_j\tau + k \tau),
  \end{equation}
for $0\le n\le \rho/\tau$. Combining this with \eqref{eq:duhamel}, we
obtain the following 
decomposition: 
 \begin{equation}\label{eq:decomp}
    Z_\tau(m_j \tau +n \tau) - \Pi_{\lambda} u(m_j \tau +n \tau)
    = \mathcal{A}_1(j) + \mathcal{A}_2(j) + \mathcal{A}_3(j) + \mathcal{A}_4(j) ,
 \end{equation}
where
    \begin{align*}
    \mathcal{A}_1(j) &:= S_{\lambda} (n\tau)\( Z_\tau(m_j \tau) - \Pi_{\lambda} u(m_j \tau)\),    \\
    \mathcal{A}_2(j) &:= S_{\lambda}(n\tau) \( \Pi_{\lambda} u (m_j \tau)-u(m_j \tau)\),
    \\
    \mathcal{A}_3(j)& := \tau \sum_{k=0}^{n-1} S_{\lambda}(n \tau -k \tau) \Big( \frac{N(\tau) -\Id}{\tau} Z_\tau(m_j\tau+k\tau)
        - \frac{N(\tau) -\Id}{\tau} \Pi_{\lambda} u(m_j\tau+k\tau) \Big),
    \\
    \mathcal{A}_4(j)&:= \tau \sum_{k=0}^{n-1} S_{\lambda}(n \tau -k
      \tau) \frac{N(\tau) -\Id}{\tau} \Pi_{\lambda} u(m_j\tau+k\tau)\\
      &\quad
+i\int_0^{n\tau} S_{\lambda}(n\tau -s) \(\vp u\) (m_j\tau+s) ds\\
&\quad
        +i\eps\int_0^{n\tau} S_{\lambda}(n\tau -s) \(|u|^{2\si} u\) (m_j\tau+s) ds,
    \end{align*}
    and we omit the dependence of the $\mathcal A_k$'s upon $n$ to
    ease notations. The goal is to
    show that in the estimates, the term $\mathcal A_3$ can be absorbed by
    the left hand side of \eqref{eq:decomp}, $\mathcal A_2$ and
    $\mathcal A_4$ are
    $\O(\tau^{1/2})$ or $\O(\tau)$, according to the case considered
    in the theorem, and $\mathcal A_1$ is then estimated by
    induction on $j$.

 Let $(q,r) \in \{(q_0, r_0), (\infty,2)\}$.
 The homogeneous Strichartz estimate \eqref{eq:discr-stri-homo} yields
    \begin{equation*}
    \| \mathcal{A}_1(j)\|_{\ell^{q}( I_j; L^r)}
    \le C_{d,q} \left\| Z_\tau(m_j \tau) - \Pi_{\lambda} u (m_j \tau) \right\|_{L^2}.
    \end{equation*}
The  term $\mathcal A_2$ is controlled again via the homogeneous discrete
Strichartz estimate  \eqref{eq:discr-stri-homo}, and
\eqref{eq:Pi-lambda-1} or \eqref{eq:Pi-lambda-1bis}: it is
$\O(\tau^{1/2})$ in the first case of Theorem~\ref{theo:stab-cv},
$\O(\tau)$ in the second case. 
 To estimate $\mathcal{A}_3$, we use \eqref{eq:NLlip}, the inhomogeneous Strichartz estimate \eqref{eq:disc-stri-inhom}, and
H\"older inequality like in \eqref{eq:holder}, to obtain
    \begin{equation*}
   \| \mathcal{A}_3(j)\|_{\ell^{q}( I_j; L^r)}
    \le C|I_j|  \|Z_\tau
    - \Pi_{\lambda} u\|_{\ell^\infty( I_j;L^2)} + C\eta^{2\si}   \|Z_\tau
    - \Pi_{\lambda} u\|_{\ell^{q_0}( I_j;L^{r_0})} ,
  \end{equation*}
 where we have used the definition of the intervals $I_j$ in terms of
  $\eta$. We now choose $\eta>0$ and $|I_j|$ sufficiently small so
  that, for all $0\le j\le K$,
  \begin{equation*}
   \| \mathcal{A}_3(j)\|_{\ell^{q_0}( I_j; L^{r_0})}+
\| \mathcal{A}_3(j)\|_{\ell^\infty( I_j; L^2)}\le \frac{1}{2}\( \|Z_\tau
    - \Pi_{\lambda} u\|_{\ell^{q_0}( I_j;L^{r_0})}+
\|Z_\tau
    - \Pi_{\lambda} u\|_{\ell^\infty( I_j;L^2)}\).
  \end{equation*}
The estimate of $\mathcal A_4$ is postponed to Lemma~\ref{lem:A4}
below,
\begin{equation}\label{eq:A4}
  \max_{0\le j\le K}\( \|\mathcal A_4(j)\|_{\ell^\infty ( I_j;L^2)}+
  \|\mathcal A_4(j)\|_{\ell^{q_0} ( I_j;L^{r_0})} \)\lesssim \tau^\alpha,
\end{equation}
where $\alpha\in \{1/2,1\}$.
Thus, we get
\begin{align*}
  \|Z_\tau
    - \Pi_{\lambda} u\|_{\ell^{q_0} (I_j;L^{r_0})}&+
  \|Z_\tau
    - \Pi_{\lambda} u\|_{\ell^\infty ( I_j;L^2)} \le  C\left\|
      Z_\tau(m_j \tau) - \Pi_{\lambda} u (m_j \tau) \right\|_{L^2}\\
      +
    C\tau^\alpha &+ \frac{1}{2} \(\|Z_\tau
    - \Pi_{\lambda} u\|_{\ell^{q_0}( I_j;L^{r_0})}+
  \|Z_\tau
    - \Pi_{\lambda} u\|_{\ell^\infty ( I_j;L^2)} \),
\end{align*}
hence, for all $0\le j\le K$, 
\begin{equation*}
  \|Z_\tau
    - \Pi_{\lambda} u\|_{\ell^{q_0}( I_j;L^{r_0})} +
  \|Z_\tau
    - \Pi_{\lambda} u\|_{\ell^\infty ( I_j;L^2)} \le 2 C\left\|
      Z_\tau(m_j \tau) - \Pi_{\lambda} u (m_j \tau) \right\|_{L^2}
    +2C\tau^\alpha .
  \end{equation*}
   Now by construction $m_0=0$, $ Z_\tau(m_0\tau) - \Pi_{\lambda} u
  (m_0 \tau) = 0$, and for $1\le j\le K$,
  \begin{equation*}
    \left\|
      Z_\tau(m_j \tau) - \Pi_{\lambda} u (m_j \tau)
    \right\|_{L^2} \le \left\| 
      Z_\tau - \Pi_{\lambda} u  \right\|_{\ell^\infty( I_{j-1};L^2)} ,
  \end{equation*}
hence, by induction,
\begin{equation*}
  \|Z_\tau
    - \Pi_{\lambda} u\|_{\ell^{q_0}([0,T];L^{r_0})} +
  \|Z_\tau
    - \Pi_{\lambda} u\|_{\ell^\infty ([0,T];L^2)} \lesssim \tau^\alpha .
  \end{equation*}
Using Strichartz estimates again, for any admissible pair $(q,r)$,
\begin{equation*}
  \|Z_\tau
    - \Pi_{\lambda} u\|_{\ell^{q}([0,T];L^{r})} \lesssim \tau^\alpha ,
  \end{equation*}
hence the first two estimates in Theorem~\ref{theo:stab-cv}. 

\subsection{Proof of  \eqref{eq:A4}}

We start with the following
lemma, which is an adaptation of \cite[Lemma~2.8]{ChoiKoh2021}. 
\begin{lemma}\label{lem:source}
Set $\lambda=1/\tau$. There exists $C$ independent of $\tau\in
(0,1)$ and the 
time interval $I$ of length at most one, 
such that
    \begin{equation*}%\label{Sov_2p+1}
    \left\| |\Pi_{\lambda} u|^{4\si+1} \right\|_{L^{q_0'} (I; L^{r_0'})}
    + \left\| |\Pi_{\lambda} u|^{2\si} \Pi_{\lambda} (|u|^{2\si} u) \right\|_{L^{q_0'} (I; L^{r_0'})}
    \le C \tau^{-1/2} \|u\|_{X^1(I)}^{4\si+1},
  \end{equation*}
  and,   
   \begin{equation*}%\label{Sov_2p+1}
    \left\| |\Pi_{\lambda} u|^{4\si+1} \right\|_{L^{q_0'} (I; L^{r_0'})}
    + \left\| |\Pi_{\lambda} u|^{2\si} \Pi_{\lambda} (|u|^{2\si} u)
    \right\|_{L^{q_0'} (I; L^{r_0'})} 
    \le C \|u\|_{X^2(I)}^{4\si+1}.
  \end{equation*}\end{lemma}
\begin{proof}
  The first case is proven in \cite[Lemma~2.8]{ChoiKoh2021}, by
  combining Bernstein inequality, H\"older inequality, and Sobolev
  embedding. We can thus mimic the proof, up to the modifications
provided by \eqref{eq:Pi-lambda-Lp}, \eqref{eq:Pi-lambda-1},
Lemmas~\ref{lem:sobolev-H}, \ref{lem:perte} and \ref{lem:bernstein},
which provide the 
analogue of \cite[Lemma~2.6]{ChoiKoh2021}.

For the second case, like in the proof of
  \cite[Lemma~2.8]{ChoiKoh2021}, we write
  \begin{equation*}
    \left\||\Pi_\tau
     u|^{4\si+1}\right\|_{L^{q_0'}L^{r_0'}}= \|\Pi_\tau
   u\|_{L^{(4\si+1)q_0'}L^{(4\si+1)r_0'}}^{4\si+1},
  \end{equation*}
and  we distinguish two cases:\\
$\bullet$ If $0<\si\le 1/d$, we use the same embedding as in
  \cite{ChoiKoh2021},
  \begin{equation*}
    H^s(\R^d) \hookrightarrow L^{(4\si+1)r_0'}(\R^d),\quad s
    =\frac{d}{2}-\frac{d(2\si+1)}{(2\si+2)(4\si+1)} \in [0,1],
  \end{equation*}
hence $X^2(I)\subset L^\infty(I;H^1)\subset
  L^{(4\si+1)q_0'}(I;L^{(4\si+1)r_0'})$. \\
$\bullet$ Now if $1/d\le  \si<2/(d-2)_+$, we define the pair $(a_2,b_2)$ by
\begin{equation*}
  a_2 = (4\si+1)q_0',\quad \frac{1}{b_2}= \frac{1}{2}-\frac{2}{da_2},
\end{equation*}
that is, we consider the Lebesgue exponent in time in the last estimate,
and pick $b_2$ so that $(a_2,b_2)$ is admissible. We check that for
such value of $\si$,
\begin{equation*}
  W^{2,b_2}(\R^d) \hookrightarrow L^{(4\si+1)r_0'}(\R^d), 
\end{equation*}
as 
\begin{equation*}
  s:= d\(\frac{1}{b_2}-\frac{1}{b_1}\) =
  \frac{d}{2}-\frac{d}{4\si+1}\(\frac{2}{d}+\frac{1}{2}\)\in [0,2]. 
\end{equation*}
The second inequality follows, using   \eqref{eq:Pi-lambda-Lp}. 
\end{proof}
The main result of this subsection is the following, in which the last
case will be used for convergence in $\cH^1$ instead of merely $L^2$:

\begin{lemma}\label{lem:A4}
Let $\lambda=1/\tau$, and $\beta=1$ or $2$.  For $u\in
X^\beta_T$, $\tau \in (0,1)$, denote 
\begin{align*}
  A(u)(n\tau)&=\tau \sum_{k=0}^{n-1} S_{\lambda}(n \tau -k \tau)
               \frac{N(\tau)-\Id}{\tau} \Pi_{\lambda} u (k \tau)+
               i\int_0^{n\tau} S_{\lambda}(n\tau -s) \vp u
               (s) ds\\
&\quad        + i\eps\int_0^{n\tau} S_{\lambda}(n\tau -s) |u|^{2\si} u (s) ds.
\end{align*}
In the case $\beta=2$, we assume in addition  $\si\ge
1/2$ (hence $d\le 5$). Then for all admissible pairs $(q,r)$ and all $T\in [0,1]$,  
\begin{equation*}
\left\|  A(u)\right\|_{\ell^q ([0,T]; L^r)} \lesssim
  \tau^{\beta/2}\(\|u\|_{X^{\beta}_T}+\|u\|_{X^{\beta}_T}^{2\si+1}+\|u\|_{X^\beta_T}^{4\si+1}\)
  . 
\end{equation*}
We also have the higher order estimate:
if $\si\ge 1/2$, 
\begin{equation*}
  \left\|  H^{1/2} A(u)\right\|_{\ell^q ([0,T]; L^r)} \lesssim
  \tau^{1/2}\(\|u\|_{X^2_T}+\|u\|_{X^2_T}^{2\si+1}+\|u\|_{X^2_T}^{4\si+1}\) .
\end{equation*}
\end{lemma}

\begin{proof}
  In the case $V=0$ and $\beta=1$, this result is exactly
  \cite[Lemma~3.1]{ChoiKoh2021}. We emphasize how this result is
  adapted to the case where $V$ satisfies Assumption~\ref{hyp:V}, and
  to the case $\beta=2$. The last case of the lemma corresponds to
  (5.19) in \cite{ChoiKoh2021}, when $V=0$. 
  
  The
  assumptions of Lemma~\ref{lem:A4} imply, by H\"older inequality and
  Sobolev embedding, 
  \begin{equation*}
    u\in L^\infty([0,T];H^1)\subset L^\theta([0,T];L^{r_0}),
  \end{equation*}
  where $\theta$ is given by \eqref{eq:theta}. Decompose
  $A(u)(n\tau)$ as 
  \[A(u)(n\tau)=A_1(u)(n\tau)+ A_2(u)(n\tau),\]
  where
\begin{equation*}
 \begin{aligned}
    A_1(u)(n\tau)&= \tau \sum_{k=0}^{n-1} S_{\lambda}(n \tau -k \tau) \mathcal{B}_1(u)(k\tau)
        - \int_0^{n\tau} S_{\lambda}(n\tau -s) \mathcal{B}_1(u)(s) ds, \\
  A_2(u)(n\tau)&=\int_0^{n\tau} S_{\lambda}(n\tau -s)
  \mathcal{B}_1(u)(s)ds+ i\int_0^{n\tau} S_{\lambda}(n\tau -s) \vp u (s)ds\\
&\quad + i\eps\int_0^{n\tau} S_{\lambda}(n\tau -s) |u|^{2\si} u (s)ds,
 \end{aligned}
\end{equation*}
with
 \[\mathcal{B}_1(u)(s) := \frac{N(\tau) -\Id}{\tau} \Pi_{\lambda}
   u(s).\]
Recall that $N_0$ denotes the nonlinear flow in the case without $\vp$,
$  N_0(t) \phi = \phi e^{-i\eps t|\phi|^{2\si}}$, 
and decompose $\mathcal{B}_1(u) =
\mathcal{B}_{2}(u)+\mathcal{B}_3(u)$, with
\begin{equation*}
  \mathcal{B}_2(u) (s)= e^{-i\tau W}  \frac{N_0(\tau)
    -\Id}{\tau} \Pi_{\lambda}u(s),\quad \mathcal{B}_3(u) (s)=  \frac{e^{-i\tau W} 
    -\Id}{\tau} \Pi_{\lambda}  u(s).
\end{equation*}
By Lemma \ref{lem-5-2}, we can estimate
\begin{align*}
  \left\| A_1(u) \right\|_{\ell^q ([0,T]; L^r)}& \lesssim
\tau^{\beta /2} \left\| H^{\beta/2}\mathcal{B}_2(u)
  \right\|_{L^{{q_0}'}_TL^{r_0'}} 
  +   \tau \left\| \partial_t \mathcal{B}_2(u) \right\|_{L^{{q_0}'}_TL^{r_0'}} \\
 &\quad+ \tau^{\beta /2} \left\| H^{\beta/2}\mathcal{B}_3(u)
  \right\|_{L^{1}_TL^{2}} 
  +   \tau \left\|\partial_t \mathcal{B}_3(u) \right\|_{L^{1}_TL^{2}} ,
\end{align*}
for all admissible pair $(q,r)$. 
In view of \eqref{eq:gradientNLLp},
H\"older inequality \eqref{eq:holder},
 and \eqref{eq:Pi-lambda-Lp}, 
   \begin{align*}
 \left\| \sqrt H \mathcal{B}_2(u) \right\|_{L^{{q_0}'}_T
  L^{r_0'}}  & \lesssim \| \Pi_{\lambda} u\|^{2\si}_{L^\theta_TL^{r_0}}
                  \|\sqrt H\Pi_{\lambda} u\|_{L^{q_0}_TL^{r_0}}
\\
    & \lesssim \| u\|^{2\si}_{L^\theta_TL^{r_0}}
  \|\Pi_\lambda \sqrt H u\|_{L^{q_0}_TL^{r_0}} \lesssim \|u\|_{X^1_T}^{2\si+1}.
   \end{align*}
In the case $\beta=2$, we invoke \eqref{eq:laplaceNLLp}: like above,
\begin{align*}
  \left\| \Pi_\lambda u\right\|^{2\si}_{L^\theta_TL^{r_0}}
  \| H \Pi_\lambda u \|_{L^{{q_0}}_T L^{{r_0}}}&
  + \| \Pi_\lambda u\|_{L^\theta_TL^{r_0}}^{2\si-1}
      \|\sqrt H \Pi_\lambda u\|_{L^{\theta}_T  L^{r_0}}
      \|\sqrt H \Pi_\lambda u\|_{L^{q_0}_T  L^{r_0}}\\
&      \lesssim \|u\|_{X^2_T}^{2\si+1}.
\end{align*}
For the last term estimating $H\mathcal{B}_2(u)$ in
\eqref{eq:laplaceNLLp}, Lemma~\ref{lem:bernstein} yields 
\begin{equation*}
 \left\| \Pi_\lambda
    u\right\|^{4\si-1}_{L^{(4\si+1)r_0'}}\left\|\sqrt H
    \Pi_\lambda u\right\|^2_{L^{(4\si+1)r_0'}}
  \lesssim \tau^{-\frac{d}{2}\(\frac{1}{r_1}-\frac{1}{r_0'}\)}
   \left\| \Pi_\lambda
    u\right\|^{4\si-1}_{L^{(4\si+1)r_1}}\left\|\sqrt H
    \Pi_\lambda u\right\|^2_{L^{(4\si+1)r_1}},
\end{equation*}
where $r_1$ is chosen like in the proof of the first
inequality in Lemma~\ref{lem:source} (see
\cite[Lemma~2.8]{ChoiKoh2021}), so we get eventually
\begin{equation*}
  \left\| H \mathcal{B}_2(u) \right\|_{L^{{q_0}'}_T
  L^{r_0'}}   \lesssim  \|u\|_{X^2_T}^{2\si+1}+ \sqrt \tau
\|u\|_{X^2_T}^{4\si+1}. 
\end{equation*}
For the other term, Lemma~\ref{lem:propNL} yields
   \begin{equation*}
     \left\| \partial_t \mathcal{B}_2(u) \right\|_{L^{q_0'}_TL^{{r_0}'}}
   = \left\| \partial_t \( \frac{N_0(\tau) -\Id}{\tau} \Pi_{\lambda} u \)
    \right\|_{L^{q_0'}_TL^{r_0'}} 
  \lesssim \left\| |\Pi_{\lambda}u|^{2\si} \partial_t \Pi_{\lambda} u
  \right\|_{L^{q_0'}_T L^{r_0'}}.
    \end{equation*}
    Observing that $\partial_t$ and $\Pi_\tau$ commute, \eqref{eq:NLSP} implies
    \begin{align*}
    \left\| \partial_t \mathcal{B}_2(u) \right\|_{L^{q_0'}_TL^{{r_0}'}}   
\lesssim  \left\|  |\Pi_{\lambda}u|^{2\si}  \Pi_{\lambda} \(H u\)
     \right\|_{L^{q_0'}_TL^{r_0'}} 
   + \left\| |\Pi_{\lambda}u|^{2\si}  \Pi_{\lambda} \(|u|^{2\si}u\)
     \right\|_{L^{q_0'}_T L^{r_0'}} .
    \end{align*}
For the first term of the right side, H\"older inequality
\eqref{eq:holder}  implies
  \begin{equation*}
   \left\|  |\Pi_{\lambda}u|^{2\si} \Pi_{\lambda}H u
     \right\|_{L^{q_0'}_TL^{r_0'}} \le  \left\|
    \Pi_{\lambda}u\right\|_{L^{\theta}_T L^{r_0}}^{2\si}
    \left\|  \Pi_{\lambda} H u \right\|_{L^{q_0}_TL^{r_0}}
  \lesssim \|u\|_{X^1_T}^{2\si} \|\Pi_{\lambda} H u
    \|_{L^{q_0}_TL^{r_0}}. 
  \end{equation*}
In the case $\beta=1$, we also invoke Lemma~\ref{lem:perte}, which
introduces an extra factor $\tau^{-1/2}$.
The second
  term of the right hand side is controlled thanks to
  Lemma~\ref{lem:source}.
For $\mathcal B_3(u)$, we easily have
\begin{equation*}
    \left\| H^{\beta/2}\mathcal{B}_3(u)
  \right\|_{L^{1}_TL^{2}} \lesssim  \| u\|_{X_T^\beta},\quad
   \left\|\partial_t \mathcal{B}_3(u) \right\|_{L^{1}_TL^{2}} \lesssim \|\Pi_{\lambda} H u
    \|_{L^\infty_TL^{2}}+ \left\| \Pi_\lambda\(|u|^{2\si}u\)\right\|_{L^1_TL^2}.
\end{equation*}
Only the last term requires some extra care: when $\beta=2$, we invoke
\eqref{eq:Pi-lambda-Lp} and the embeddings $H^2\subset
H^{d\si/(2\si+1)}\subset L^{4\si+2}$ used in Section~\ref{sec:cauchyH2}, 
\begin{equation*}
  \left\| \Pi_\lambda\(|u|^{2\si}u\)\right\|_{L^1_TL^2}\lesssim
  \|u\|^{2\si+1}_{L^\infty_TL^{4\si+2}}\lesssim
  \|u\|_{X_T^2}^{2\si+1}. 
\end{equation*}
In the case $\beta=1$, we invoke Lemma~\ref{lem:bernstein}: for $1\le
p\le 2$,
\begin{equation*}
   \left\| \Pi_\lambda\(|u|^{2\si}u\)\right\|_{L^1_TL^2}\lesssim
   \tau^{-\frac{d}{2}\(\frac{1}{p}-\frac{1}{2}\)}\|u\|^{2\si+1}_{L^\infty_TL^{(2\si+1)p}}. 
\end{equation*}
If $\si\le 1/d$, we pick $(2\si+1)p=2$, and
\begin{equation*}
   \left\| \Pi_\lambda\(|u|^{2\si}u\)\right\|_{L^1_TL^2}\lesssim
   \tau^{-\frac{d\si}{2}}\|u\|^{2\si+1}_{L^\infty_TL^{2}}\lesssim
   \tau^{-\frac{1}{2}}\|u_0\|^{2\si+1}_{L^{2}}. 
\end{equation*}
If $\si>1/d$, we use the Sobolev embedding $H^s(\R^d)\hookrightarrow
L^{\frac{2d}{d+2}(2\si+1)} (\R^d)$ where
\begin{equation*}
  s= \frac{d\si-1}{2\si+1}<1,
\end{equation*}
and this Lebesgue index is chosen so that the power of $\tau$ is
exactly $-1/2$,
\begin{equation*}
  \left\| \Pi_\lambda\(|u|^{2\si}u\)\right\|_{L^1_TL^2}\lesssim
  \tau^{-1/2}  \|u\|^{2\si+1}_{L^\infty_TH^s}\lesssim
  \tau^{-1/2}  \|u\|^{2\si+1}_{X^1_T}.
\end{equation*}
 We come up with
  \begin{equation*}
   \left\| A_1(u) \right\|_{\ell^q ([0,T]; L^r)}\lesssim
    \tau^{\beta/2}\(\|u\|_{X^\beta_T}+
    \|u\|_{X^\beta_T}^{2\si+1}+\|u\|_{X^\beta_T}^{4\si+1}\).  
  \end{equation*}
  To complete the proof, we perform another decomposition, for the
   term $A_2$,
    \[
     \mathcal{B}_1(u) +i\vp u+i \eps|u|^{2\si} u= \mathcal C_1(u)+ \mathcal{C}_2(u) + \mathcal{C}_3(u) +i\mathcal{C}_4(u)+i\eps\mathcal{C}_5(u),\]
   where
  \begin{align*}
&\mathcal C_1(u) = \( e^{-i\tau \vp}-1\) \frac{N_0(\tau)-\Id}{\tau}\Pi_\lambda u,\\
&\mathcal{C}_2(u) = \frac{N_0(\tau)-\Id}{\tau}\Pi_\lambda u
    +i\eps|\Pi_{\lambda}u|^{2\si} \Pi_{\lambda}u,    \\
&\mathcal C_3(u) = \frac{ e^{-i\tau \vp}-\Id}{\tau}\Pi_{\lambda}u
+i\vp  \Pi_\lambda    u,\\
& \mathcal{C}_4(u) =\vp (u-\Pi_\lambda u),\\
&\mathcal C_5(u)=|u|^{2\si} u-|\Pi_{\lambda}u|^{2\si} \Pi_{\lambda}u.
  \end{align*}
The discrete inhomogeneous Strichartz estimates \eqref{eq:disc-stri-inhom} yields,
for  $(q,r)$ admissible,
\begin{align*}
      \left\|\mathcal{A}_2(u) \right\|_{\ell^q ([0,T]; L^r)}
      &\lesssim \left\| \mathcal{C}_1(u) \right\|_{L^{q_0'}_T L^{r_0'}}
        + \left\| \mathcal{C}_2(u) \right\|_{L^{q_0'}_T L^{r_0'}}
+\left\| \mathcal{C}_3(u) \right\|_{L^1_T L^2}+\left\|
  \mathcal{C}_4(u) \right\|_{L^1_T L^2}\\
&\quad+ \left\| \mathcal{C}_5(u) \right\|_{L^{q_0'}_T L^{r_0'}}. 
\end{align*}
In view of \eqref{eq:NLpuissance}, we have the pointwise estimate
\begin{equation*}
 | \mathcal C_1(u)|\lesssim \tau\|\vp\|_{L^\infty} |\Pi_\lambda u|^{2\si+1}, 
\end{equation*}
hence 
\begin{equation*}
  \left\| \mathcal{C}_1(u) \right\|_{L^{q_0'}_T L^{r_0'}}\lesssim \tau
  \|u\|_{X_T^1}^{2\si+1}. 
\end{equation*}
Recall that $N_0(\tau)z= ze^{-i\eps\tau|z|^{2\si}}$, Taylor formula yields the pointwise estimate
\[
    \left| \mathcal{C}_2(u) \right|
    = \Big|\frac{N_0(\tau) -\Id}{\tau} \Pi_{\lambda} u +i \eps|\Pi_{\lambda}u|^{2\si} \Pi_{\lambda}u \Big|
      \lesssim \tau |\Pi_{\lambda}u|^{4\si+1}.
\]
Thus, we have
    \begin{equation*}
    \left\| \mathcal{C}_2(u) \right\|_{L^{q_0'}_T L^{r_0'}}
    \lesssim \tau \left\| |\Pi_{\lambda} u|^{4\si+1} \right\|_{L^{q_0'}_T
    L^{r_0'}},
  \end{equation*}
and this last term is estimated by Lemma~\ref{lem:source}. We readily
have
\begin{equation*}
  \left\| \mathcal{C}_3(u) \right\|_{L^1_T L^2}\lesssim \tau
  \|\Pi_\lambda u\|_{L^\infty_TL^2}\lesssim \tau\|u_0\|_{L^2},\quad
  \left\| \mathcal{C}_4(u) \right\|_{L^1_T L^2}\lesssim
  \tau^{\beta/2}\|u\|_{X^\beta_T}. 
\end{equation*}
Finally, H\"older inequality \eqref{eq:holder} yields
    \begin{align*}
    \left\| \mathcal{C}_5(u) \right\|_{L^{q_0'}_TL^{r_0'}}
    &=\left\| |\Pi_{\lambda}u|^{2\si} \Pi_{\lambda}u - |u|^{2\si} u
      \right\|_{L^{q_0'}_TL^{r_0'}} \\
      &\lesssim \(\| \Pi_{\lambda}u \|_{L^\theta_TL^{r_0}}^{2\si} +
        \|u \|_{L^\theta_TL^{r_0}}^{2\si}\)
      \left\| \Pi_{\lambda}u - u \right\|_{L^{q_0}_T L^{r_0}}
\lesssim \tau^{\beta/2} \|u\|_{X^\beta_T}^{2\si+1} ,
    \end{align*}
where we have used Lemma~\ref{lem:PiPDO}. This yields the announced
estimate for $A (u)$ in $\ell^{q}([0,T];L^{r})$. We now
emphasize the modifications needed to estimate $\sqrt H  A (u)$
in the same space, and thus conclude the proof of
Lemma~\ref{lem:A4}.  Invoking discrete
Strichartz estimates like above, we have 
\begin{align*}
    \left\| \sqrt HA_1(u) \right\|_{\ell^q ([0,T]; L^r)} &\lesssim
  \tau^{1/2} \left\| H\mathcal{B}_2(u)
  \right\|_{L^{{q_0}'}_TL^{r_0'}}   +   \tau \left\| \sqrt H \partial_t \mathcal{B}_2(u)
   \right\|_{L^{{q_0}'}_TL^{r_0'}} \\
&\quad +  \tau^{1/2} \left\| H\mathcal{B}_3(u)
  \right\|_{L^1_TL^2}   +   \tau \left\| \sqrt H \partial_t \mathcal{B}_3(u)
   \right\|_{L^1_TL^2}.
\end{align*}
The previous computations, in the case $\beta=2$, can then be resumed
with essentially no modification, except maybe for the  term
  involving $ \partial_t \mathcal{B}_2(u)$. The factor $e^{-i\tau\vp}$ is
  obviously discarded, and $\partial_t \mathcal{B}_2(u)$ appears as a linear combination
of the terms 
\begin{equation*}
  \frac{N_0(\tau)-\Id}{\tau} |\Pi_\lambda u|^{2\si} \partial_t \Pi_\lambda
  u\quad \text{and}\quad    \frac{N_0(\tau)-\Id}{\tau} |\Pi_\lambda
  u|^{2\si-2}\(\Pi_\lambda u\)^2 \partial_t \overline{\Pi_\lambda u}.
\end{equation*}
Like before, we use the commutation between $\partial_t$ and $\Pi_\lambda$,
and \eqref{eq:NLSP}. Invoking Lemma~\ref{lem:H12N0} and
\eqref{eq:H12der}, the most delicate terms to control are
\begin{align*}
 I=& \|\Pi_\lambda u\|^{2\si}_{L^\theta_TL^{r_0}}\left\|
  \sqrt H \Pi_\lambda H u\right\|_{L^{q_0}_T L^{r_0}} ,\\
  II=& \|\Pi_\lambda u\|_{L^\theta_TL^{r_0}}^{2\si-1}
       \|\sqrt H \Pi_\lambda u\|_{L^\theta_TL^{r_0}}
       \|\Pi_\lambda H u\|_{L^{q_0}_T L^{r_0}}, \\
      & \quad \text{when   differentiating the power of } \Pi_\lambda
        u,\\ 
 III=&\|\Pi_\lambda u\|_{L^{2\theta}_TL^{2r_0}}^{4\si-1} \|\sqrt H
  \Pi_\lambda u\|_{L^{2\theta}_TL^{2r_0}} \|\Pi_\lambda H
  u\|_{L^{q_0}_T L^{r_0}},\\
  & \quad      \text{when differentiating }\frac{N_0(\tau)-\Id}{\tau}\text{ again}.
\end{align*}
Proceeding like before, Lemma~\ref{lem:perte} and
\eqref{eq:Pi-lambda-Lp}yield 
\begin{equation*}
  I \lesssim
 \tau^{-1/2} \|u\|_{X_T^2}^{2\si+1} .
\end{equation*}
For the term $II$, when $1/2\le \si<2/d$, we write
\begin{equation*}
  \|\sqrt H \Pi_\lambda u\|_{L^\theta_TL^{r_0}}\le T^{1/\theta-1/q_0}
  \|\sqrt H \Pi_\lambda u\|_{L^{q_0}_TL^{r_0}} 
\lesssim \|u\|_{X_T^1}. 
\end{equation*}
When $\si\ge 2/d$, we resume the estimate from the proof of
Proposition~\ref{prop:cauchy2}, with $(\theta,b_1)$ admissible,
\begin{equation*}
   \|\sqrt H \Pi_\lambda u\|_{L^\theta_TL^{r_0}}\lesssim \|H \Pi_\lambda
  u\|_{L^\theta_TL^{b_1}} \lesssim \|u\|_{X_T^2}. 
\end{equation*}
We obtain, in both cases,
\begin{equation*}
  II\lesssim \|u\|_{X_T^2}^{2\si+1} .
\end{equation*}
The last term is estimated differently from  the proof of
Lemma~\ref{lem:source}: 
\begin{equation*}
  III\lesssim \|\Pi_\lambda u\|_{L^{2\theta}_TL^{2r_0}}^{4\si-1}
  \|\sqrt H 
  \Pi_\lambda u\|_{L^{2\theta}_TL^{2r_0}} \|u\|_{X_T^2}. 
\end{equation*}
We check that since $\si<2/(d-2)_+$ and $d\le 5$,
$H^2(\R^d)\hookrightarrow L^{2r_0}(\R^d)=L^{4\si+4}(\R^d)$, 
  \begin{equation*}
    \|\Pi_\lambda u\|_{L^{2\theta}_TL^{2r_0}}\lesssim
    \|u\|_{L^{2\theta}_TL^{2r_0}}\lesssim \|u\|_{L^\infty_T
      H^2}\lesssim \|u\|_{X_T^2},
  \end{equation*}
and, using Lemma~\ref{lem:perte},
\begin{equation*}
   \|\Pi_\lambda\sqrt H u\|_{L^\infty_T H^2}\lesssim \tau^{-1/2}
   \|u\|_{X_T^2}.
\end{equation*}
We leave out the details of the other estimates. 
\end{proof}

\subsection{Stability implies convergence in $\cH^1$}

We now get to the last estimate in Theorem~\ref{theo:stab-cv}. We
therefore assume $1/2\le \si<2/(d-2)_+$. Again,
in view of Lemma~\ref{lem:PiPDO}, it suffices to consider the 
difference $Z_\tau(n\tau)-\Pi_\lambda u(n\tau)$.
The scheme is essentially the same as in Subsection~\ref{sec:cv-L2}:
we decompose $[0,T]$ into a finite union of intervals $I_j$, on which
some convenient norms will be sufficiently small so the corresponding
term on the right hand side is absorbed by the left hand side, when
using Strichartz estimates. The core of the proof of convergence in
$\cH^1$ then relies on higher order estimates of the terms involved in
Subsection~\ref{sec:cv-L2}. 
\smallbreak

We resume
Duhamel formula \eqref{eq-a-1'} and the decomposition
\eqref{eq:decomp}. The terms $\mathcal A_1$ and $\mathcal A_2$ are
estimated thanks to the homogeneous discrete Strichartz estimate, 
\begin{align*}
   \|\sqrt H
  \mathcal A_1(j)\|_{\ell^q( I_j;L^r)} &= \left\|S_\lambda \sqrt H\(
  Z_\tau(m_j \tau)- \Pi_\lambda u(m_j \tau)\) \right\|_{\ell^q( I_j;L^r)}\\
 &\lesssim  \left\|\sqrt H\(
  Z_\tau(m_j \tau)- \Pi_\lambda u(m_j \tau)\) \right\|_{L^2}\\
&\lesssim \left\|
  Z_\tau(m_j \tau)- \Pi_\lambda u(m_j \tau) \right\|_{\cH^1},
\end{align*}
and the term  $\mathcal A_2$ is estimated similarly, 
\begin{align*}
   \|\sqrt H \mathcal A_2(j)\|_{\ell^q(I_j;L^r)}&\lesssim  \left\|
\sqrt H\(  \Pi_\lambda u(m_j \tau)-  u(m_j \tau)\)
   \right\|_{L^2}\\
  &\lesssim \left\|
\Pi_\lambda \sqrt Hu(m_j \tau)-  \sqrt Hu(m_j \tau) \right\|_{L^2} \\
  &\lesssim \tau^{1/2} \|u(m_j\tau)\|_{\cH^2}\lesssim \tau^{1/2} \|u\|_{X_T^2},
\end{align*}
where we have used \eqref{eq:Pi-lambda-1}. The term $\mathcal A_4$ is
estimated thanks 
to the last case of Lemma~\ref{lem:A4}, 
\begin{equation*}
  \|\sqrt H \mathcal A_4(j)\|_{\ell^q(I_j;L^r)}\lesssim \tau^{1/2} \(
  \|u\|_{X_T^2}+ \|u\|_{X_T^2}^{2\si+1}+  \|u\|_{X_T^2}^{4\si+1}\). 
\end{equation*}
Therefore, the new estimate needed at this stage is the estimate of
$\mathcal A_3$. To simplify the presentation, we treat the case
$\vp=0$ only, as we have seen before how to handle the presence of
this term. 
We prove the analogue of \cite[Lemma~5.2]{ChoiKoh2021}:
\begin{lemma}\label{lem:5.2}
  Let $1/2\le \si<2/(d-2)_+$.
  There exists $C$ such that for all time
interval $I$ and all $w_1,w_2$ sufficiently regular,
\begin{align*}
  & \left\|\sqrt H\(\frac{N_0(\tau)-\Id}{\tau}w_1 -
    \frac{N_0(\tau)-\Id}{\tau}w_2\)\right\|_{\ell^{q_0'}(I;L^{r_0'})} \\
  &\quad\le C 
    \|w_1\|^{2\si}_{\ell^\theta(I;L^{r_0})}
    \|\sqrt H(w_1-w_2)\|_{\ell^{q_0}(I;L^{r_0})}\\
  &\qquad + C\Bigl(
    \tau\|w_1\|^{4\si-1}_{\ell^{\frac{4\si-1}{2\si-1}\theta}(I;L^{\frac{4\si-1}{2\si-1}r_0})}+
    \tau\|w_2\|^{4\si-1}_{\ell^{\frac{4\si-1}{2\si-1}\theta}(I;L^{\frac{4\si-1}{2\si-1}r_0})} 
     \\
   &\qquad \qquad+ 
    \|w_1\|^{2\si-1}_{\ell^\theta(I;L^{r_0})}+
    \|w_2\|^{2\si-1}_{\ell^\theta(I;L^{r_0})}\Bigr)
     \|\sqrt H
    w_2\|_{\ell^\theta(I;L^{r_0})}\|w_1-w_2\|_{\ell^{q_0}(I;L^{r_0})}\\
  &\qquad + C\tau\|w_1\|^{4\si-1}_{\ell^{\frac{4\si-1}{2\si-1}\theta}(I;L^{\frac{4\si-1}{2\si-1}r_0})}\|w_1\|_{\ell^\theta(I;L^{r_0})} \|\sqrt H(w_1-w_2)\|_{\ell^{q_0}(I;L^{r_0})}.
\end{align*}
\end{lemma}
\begin{proof}
  When $\sqrt H$ is replaced by $J\in \{\sqrt V,\Id\}$ on the left
  hand side, the result is a direct consequence of
  \eqref{eq:NLlip} and H\"older inequality \eqref{eq:holder},
  recalling Lemma~\ref{lem:sobolev-H}. We thus consider the case
  $J=\nabla$, and compute
\begin{align*}
\nabla &\( \frac{N_0(\tau)-\Id}{\tau}w_1 -
         \frac{N_0(\tau)-\Id}{\tau}w_2\) \\
       & = \frac{e^{-i\eps \tau |w_1|^{2\si}}-1}{\tau} \nabla w_1 -
         \frac{e^{-i\eps \tau |w_2|^{2\si}}-1}{\tau} \nabla w_2
         \pm \frac{e^{-i\eps \tau |w_1|^{2\si}}-1}{\tau} \nabla w_2\\
  &\quad -i\tau\eps w_1\nabla |w_1|^{2\si} \frac{e^{-i\eps \tau
    |w_1|^{2\si}}-1}{\tau}
    +i\tau\eps w_2\nabla |w_2|^{2\si} \frac{e^{-i\eps \tau
    |w_2|^{2\si}}-1}{\tau}\\
  &\quad \pm i\tau\eps w_2\nabla |w_2|^{2\si} \frac{e^{-i\eps \tau
    |w_1|^{2\si}}-1}{\tau},
\end{align*}
where the symbol ``$\pm$'' means that we add and subtract the same
term in order to group suitable terms together. We have, in view of
\eqref{eq:NLpuissance},
\begin{equation*}
  \left| \frac{e^{-i\eps \tau |w_1|^{2\si}}-1}{\tau}
    \nabla(w_1-w_2)\right| \le |w_1|^{2\si}|\nabla (w_1-w_2)|,
\end{equation*}
and \eqref{eq:holder} yields, together with Lemma~\ref{lem:sobolev-H},
\begin{align*}
  \left\| \frac{e^{-i\eps \tau |w_1|^{2\si}}-1}{\tau}
    \nabla(w_1-w_2)\right\|_{\ell^{q_0'}(I;L^{r_0'})} &\le
  \|w_1\|^{2\si}_{\ell^\theta(I;L^{r_0})}
  \|\nabla (w_1-w_2)\|_{\ell^{q_0}(I;L^{r_0})}\\
&  \lesssim \|w_1\|^{2\si}_{\ell^\theta(I;L^{r_0})}
  \|\sqrt H (w_1-w_2)\|_{\ell^{q_0}(I;L^{r_0})}.
\end{align*}
We next write, like for \eqref{eq:NLlip},
\begin{align*}
  \left| \(\frac{e^{-i\eps \tau |w_1|^{2\si}}-1}{\tau}-\frac{e^{-i\eps
  \tau |w_2|^{2\si}}-1}{\tau}\)\nabla w_2\right|&= \left|
  \(\frac{e^{-i\eps \tau |w_1|^{2\si}}-e^{-i\eps 
   \tau |w_2|^{2\si}}}{\tau}\)\nabla w_2\right|\\
 \lesssim &\left||w_1|^{2\si}-   |w_2|^{2\si}\right| |\nabla w_2|\\
 \lesssim &\(|w_1|^{2\si-1}+   |w_2|^{2\si-1}\)|w_1-w_2|\times |\nabla w_2|, 
\end{align*}
and \eqref{eq:holder} yields
\begin{align*}
 & \left\| \(\frac{e^{-i\eps \tau |w_1|^{2\si}}-1}{\tau}-\frac{e^{-i\eps
  \tau |w_2|^{2\si}}-1}{\tau}\)\nabla
  w_2\right\|_{\ell^{q_0'}(I;L^{r_0'})}\\
  &\qquad \lesssim
    \(\|w_1\|^{2\si-1}_{\ell^\theta(I;L^{r_0})}
    +\|w_2\|^{2\si-1}_{\ell^\theta(I;L^{r_0})}\)
    \|\nabla w_2\|_{\ell^\theta(I;L^{r_0})}
    \|w_1-w_2\|_{\ell^{q_0}(I;L^{r_0})}.
\end{align*}
Writing in a rather loose fashion, that is, proceeding as if
$w_1,w_2\ge 0$,
\begin{equation*}
  w_1\nabla |w_1|^{2\si} - w_2\nabla |w_2|^{2\si} \approx 2\si\(
  w_1^{2\si}\nabla w_1 - w_2^{2\si}\nabla w_2\pm w_1^{2\si}\nabla w_2\),
\end{equation*}
we also have, in view of \eqref{eq:NLpuissance},
\begin{align*}
  \left| \(w_1\nabla |w_1|^{2\si} - w_2\nabla |w_2|^{2\si} \)
  \frac{e^{-i\eps \tau |w_1|^{2\si}}-1}{\tau}\right|& \lesssim
 |w_1|^{4\si} |\nabla(w_1-w_2)| \\
  + |w_1|^{2\si} &\( |w_1|^{2\si-1}+
  |w_2|^{2\si-1}\)|w_1-w_2| |\nabla w_2|. 
\end{align*}
Finally, like for \eqref{eq:NLlip},
\begin{align*}
  &\left| \(\frac{e^{-i\eps \tau |w_1|^{2\si}}-1}{\tau}-\frac{e^{-i\eps
  \tau |w_2|^{2\si}}-1}{\tau}\)w_2\nabla |w_2|^{2\si}\right|\\
  &\qquad \lesssim  \left||w_1|^{2\si}-|w_2|^{2\si}\right|  |w_2|^{2\si}|\nabla w_2|,
\end{align*}
so the sum of the last two terms considered is controlled by
\begin{equation*}
  |w_1|^{4\si} |\nabla(w_1-w_2)|
  +\( |w_1|^{4\si-1}+
  |w_2|^{4\si-1}\)|w_1-w_2| |\nabla w_2|.
\end{equation*}
The $\ell^{q_0'}(I;L^{r_0'})$-norm of this quantity is estimated, in
view of \eqref{eq:holder}, by
\begin{align*}
  &\|w_1\|_{\ell^{\frac{4\si-1}{2\si-1}\theta}(I;L^{\frac{4\si-1}{2\si-1}r_0})}^{4\si-1}
  \|w_1\|_{\ell^\theta(I;L^{r_0})}\|\nabla
  (w_1-w_2)\|_{\ell^{q_0}(I;L^{r_0})}\\
&  +\(\sum_{j=1}^2\|w_j\|_{\ell^{\frac{4\si-1}{2\si-1}\theta}(I;L^{\frac{4\si-1}{2\si-1}r_0})}^{4\si-1}
\)
 \|\nabla w_2\|_{\ell^\theta(I;L^{r_0})}\|
  w_1-w_2\|_{\ell^{q_0}(I;L^{r_0})},
\end{align*}
hence the lemma. 
\end{proof}
We apply Lemma~\ref{lem:5.2} to $\mathcal A_3$, with $w_1 = Z_\tau$ and
$w_2=\Pi_\lambda u$. We have, thanks to Strichartz estimates
\eqref{eq:disc-stri-inhom},   \eqref{eq:H12der} and H\"older inequality, 
\begin{align*}
  \|&\sqrt H\mathcal A_3(j) \|_{\ell^q(I_j;L^{r})}\lesssim \|Z_\tau\|^{2\si}_{\ell^\theta(I_j;L^{r_0})}
    \|\sqrt H(Z_\tau-\Pi_\lambda u)\|_{\ell^{q_0}(I_j;L^{r_0})}\\
  &+\Bigl(
    \tau\|Z_\tau\|^{4\si-1}_{\ell^{\frac{4\si-1}{2\si-1}\theta}(I_j;L^{\frac{4\si-1}{2\si-1}r_0})}+
    \tau\|\Pi_\lambda u\|^{4\si-1}_{\ell^{\frac{4\si-1}{2\si-1}\theta}(I_j;L^{\frac{4\si-1}{2\si-1}r_0})}
     \\
   & \qquad+ 
    \|Z_\tau\|^{2\si-1}_{\ell^\theta(I_j;L^{r_0})}+
    \|\Pi_\lambda u\|^{2\si-1}_{\ell^\theta(I_j;L^{r_0})}\Bigr)
     \|\sqrt H
    \Pi_\lambda u\|_{\ell^\theta(I_j;L^{r_0})}\|Z_\tau-\Pi_\lambda u\|_{\ell^{q_0}(I_j;L^{r_0})}\\
  &+\tau\|Z_\tau\|^{4\si-1}_{\ell^{\frac{4\si-1}{2\si-1}\theta}(I_j;L^{\frac{4\si-1}{2\si-1}r_0})}\|Z_\tau\|_{\ell^\theta(I_j;L^{r_0})}
    \|\sqrt H(Z_\tau-\Pi_\lambda u)\|_{\ell^{q_0}(I_j;L^{r_0})}\\
    &\phantom{\sqrt H\mathcal A_3 \|_{\ell^q(I_j;L^{r})}}
      \lesssim |I_j|^{2\si\gamma}
    \|\sqrt H(Z_\tau-\Pi_\lambda u)\|_{\ell^{q_0}(I_j;L^{r_0})}\\
  &+\Bigl(
    \tau\|Z_\tau\|^{4\si-1}_{\ell^{\frac{4\si-1}{2\si-1}\theta}(I_j;L^{\frac{4\si-1}{2\si-1}r_0})}+
    \tau\|\Pi_\lambda u\|^{4\si-1}_{\ell^{\frac{4\si-1}{2\si-1}\theta}(I_j;L^{\frac{4\si-1}{2\si-1}r_0})}+1\Bigr)
     \\
   & \qquad\qquad \times |I_j|^{1/\theta}
     \|\sqrt H
     u\|_{\ell^\infty(I_j;H^1)}\|Z_\tau-\Pi_\lambda u\|_{\ell^{q_0}(I_j;L^{r_0})}\\
  &+\tau\|Z_\tau\|^{4\si-1}_{\ell^{\frac{4\si-1}{2\si-1}\theta}(I_j;L^{\frac{4\si-1}{2\si-1}r_0})}|I_j|^\gamma M
    \|\sqrt H(Z_\tau-\Pi_\lambda u)\|_{\ell^{q_0}(I_j;L^{r_0})} ,
\end{align*}
where we have used \eqref{eq:stab-num}. Noticing that the assumptions
on  $\chi$ imply
\begin{equation*}
\chi(z)= \chi\(\frac{z}{4}\)\chi(z),\quad \text{hence}\quad  \Pi_\lambda =
\Pi_{4\lambda}\Pi_\lambda , 
\end{equation*}
we have
\begin{equation*}
  Z_\tau = \Pi_{4\lambda}Z_\tau.
\end{equation*}
To control the $L^{\frac{4\si-1}{2\si-1}r_0}$-norms in space, we now invoke
Lemma~\ref{lem:bernstein}, with
\begin{equation*}
  p=r_0<q=\frac{4\si-1}{2\si-1}r_0,
\end{equation*}
applied to  $\phi=Z_\tau$ or $\Pi_\lambda u$, so we have $\phi=
\Pi_{4\lambda}\phi$ (leaving out the time dependence at this stage),
and thus 
\begin{equation*}
  \|f\|_{L^{\frac{4\si-1}{2\si-1}r_0}}^{4\si-1}\lesssim \tau^{-\frac{d}{2}\(
 \frac{1}{r_0}-\frac{2\si-1}{(4\si-1)r_0}   \)(4\si-1)}\|f\|_{L^{r_0}}^{4\si-1}. 
\end{equation*}
We compute
\begin{equation*}
  \frac{d}{2}\(
 \frac{1}{r_0}-\frac{2\si-1}{(4\si-1)r_0}   \)(4\si-1) =
 \frac{d\si}{2\si+2}\le 1, \quad\text{since}\quad\si<\frac{2}{(d-2)_+},
\end{equation*}
so the loss induced by this Bernstein-type inequality is compensated
by the factor $\tau$ in front of the corresponding norm,
\begin{equation*}
  \tau\|Z_\tau\|^{4\si-1}_{\ell^{\frac{4\si-1}{2\si-1}\theta}(I_j;L^{\frac{4\si-1}{2\si-1}r_0})}+
    \tau\|\Pi_\lambda
    u\|^{4\si-1}_{\ell^{\frac{4\si-1}{2\si-1}\theta}(I_j;L^{\frac{4\si-1}{2\si-1}r_0})}\lesssim
    M^{4\si-1} + \|u\|_{X_T^1}^{4\si-1},
\end{equation*}
where we have used \eqref{eq:stab-num2} and the Sobolev embedding $H^1(\R^d)\hookrightarrow
L^{r_0}(\R^d)$.   
We can then
simplify the above estimate to
\begin{align*}
  \|\sqrt H&\mathcal A_3 (j)\|_{\ell^q(I_j;L^{r})}
      \lesssim |I_j|^{2\si\gamma}
    \|\sqrt H(Z_\tau-\Pi_\lambda u)\|_{\ell^{q_0}(I_j;L^{r_0})}\\
  & +|I_j|^{1/\theta}
   \(M^{4\si-1} + \|u\|_{X_T^1}^{4\si-1}  +1\)\|
     u\|_{X_T^2}\|Z_\tau-\Pi_\lambda u\|_{\ell^{q_0}(I_j;L^{r_0})}\\
  &+|I_j|^\gamma M^{4\si}
    \|\sqrt H(Z_\tau-\Pi_\lambda u)\|_{\ell^{q_0}(I_j;L^{r_0})} .
\end{align*}
We can then
proceed like in Subsection~\ref{sec:cv-L2}, and leave out
the details.

\section{Stability}
\label{sec:stability}

To complete the proof of Theorems~\ref{theo:main} and
\ref{theo:main2}, we now have to prove that the stability conditions
\eqref{eq:stab-num}, and \eqref{eq:stab-num2} in the corresponding
case, are
verified.  Throughout this section, $T>0$ is fixed, such that $u\in
C([0,T];\cH^1)$. We distinguish two cases, for which the roadmap is the
same, but the technical details are different. Indeed, we emphasize
the property
\begin{equation*}
  \theta<q_0\Longleftrightarrow \si<\frac{2}{d}.
\end{equation*}
In the $L^2$-subcritical case $0<\si<2/d$, we use the H\"older inequality
\begin{equation*}
  \|u^n\|_{\ell^\theta(I;L^{r_0})}\le |I|^{1/\theta-1/q_0}
  \|u^n\|_{\ell^{q_0}(I;L^{r_0})}, 
\end{equation*}
and the goal is to prove that $\|u^n\|_{\ell^{q_0}([0,T];L^{r_0})}\le
M$ for some $M$ independent of $\tau\in (0,1)$.
\smallbreak

If $\si\ge 2/d$, we use Sobolev embedding and H\"older inequality,
\begin{align*}
  &\|u^n\|_{\ell^\theta(I;L^{r_0})}\lesssim \|u^n\|_{\ell^\theta(I;H^1)}
    \lesssim |I|^{1/\theta}\|u^n\|_{\ell^\infty(I;H^1)} ,\\
 &\|u^n\|_{\ell^\infty(I;L^{r_0})}\lesssim \|u^n\|_{\ell^\infty(I;H^1)} , 
\end{align*}
and the goal is to prove that $\|u^n\|_{\ell^\infty([0,T];H^1)}\le
M$ for some $M$ independent of $\tau\in (0,1)$.

\subsection{$L^2$-subcritical case: $0<\si<2/d$}

In the absence of potential, $V=0$, \eqref{eq:stab-num} was proven
in
\cite[Theorem~1.1]{Ignat2011} and resumed in
\cite[Proposition~4.1]{ChoiKoh2021}.  We slightly modify the argument
here, in a lemma where we do not try to make the constants sharp:
\begin{lemma}\label{lem:stab-sub} 
  Let $0<\si<2/d$, $u_0\in \cH^1$ and $T>0$ such that $u\in X_T^1$. 
There exists $\tau_0>0$ such that for all $\tau\in (0,\tau_0]$,
$Z_\tau(n\tau)u_0\in \ell^{q_0}([0,T];L^{r_0})$, and there exists
$C>0$ such that 
\begin{equation*}
  \|Z_\tau(n\tau)u_0\|_{\ell^{q_0}([0,T];L^{r_0})}\le C\|u\|_{L^\infty_T\cH^1}.
\end{equation*}
\end{lemma}
\begin{proof}
The idea of the proof is a bootstrap argument: as long as
$Z_\tau(k\tau)u_0$ is bounded in $\ell^{q_0}L^{r_0}$,
Theorem~\ref{theo:stab-cv} implies that its size is the same as that
of $u$ modulo $\O(\tau^{1/2})$, hence a uniform bound for $\tau>0$
sufficiently small.
Let
$\Lambda=\Lambda(\tau)$ defined by
\begin{align*}
  \Lambda &=\Big\{N\in \N,
   \|Z_\tau(k\tau)u_0\|_{\ell^{q_0}(0\le n\tau\le N\tau;L^{r_0})}\le
    (\uC+C_1)\|u_0\|_{L^2} \\
&\phantom{=\{N\in \N,
   \|Z_\tau(k\tau)u_0\|_{\ell^{q_0}(0\le n\tau\le N\tau;L^{r_0})}\le }
+  C_1 \|u(n\tau)\|_{\ell^{q_0}([0,T];L^{r_0})}\Big\},
\end{align*}
where $\uC=\uC(d,\si,T)$ is the best constant provided by the homogeneous
discrete Strichartz estimate \eqref{eq:discr-stri-homo} in the case
$\lambda=1/\tau$, on the time interval $[0,T]$,
\begin{equation*}
  \uC = \sup_{\tau\in (0,1)}\sup_{\phi\in
    L^2(\R^d)}\frac{\|S_{1/\tau}\phi\|_{\ell^{q_0}([0,T];L^{r_0})}}{\|\phi\|_{L^2}}, 
\end{equation*} 
and $C_1>0$ is determined below. Note that Sobolev embedding and the
finiteness of $T$ yield (even though this estimate is far from being sharp)
\begin{equation*}
  \|u(n\tau)\|_{\ell^{q_0}([0,T];L^{r_0})}\lesssim \|u\|_{L^\infty_T\cH^1}.
\end{equation*}
First, regardless of the value of $C_1$, $\Lambda$ is not empty, as
$0\in \Lambda$: 
\begin{equation*}
  \tau^{1/q_0}\|Z_\tau (0)u_0\|_{L^{r_0}} =
  \tau^{1/q_0}\|S_{1/\tau}(0)u_0\|_{L^{r_0}} \le \uC
  \|u_0\|_{L^2}. 
\end{equation*}
If $\Lambda$ is infinite, then \eqref{eq:stab-num} holds, in view of
the above presentation. If $\Lambda$
is bounded, let $N_*=N_*(\tau)$ be its largest element: we suppose
that 
 $N_*+1< T/\tau$.
Like in the proof of Lemma~\ref{lem:A4}, we decompose
\begin{equation*}
  \frac{N(\tau)-\Id}{\tau}Z_\tau
    (n\tau)u_0 = e^{-i\tau\vp} \frac{N_0(\tau)-\Id}{\tau}Z_\tau
    (n\tau)u_0  + \frac{ e^{-i\tau\vp} -\Id}{\tau}Z_\tau
    (n\tau)u_0  .
\end{equation*}
 In view of the
discrete Duhamel's formula (see 
\eqref{eq-a-1'}), discrete Strichartz estimates yield
\begin{align*}
  \(\tau\sum_{k=0}^{N_*+1}\|Z_\tau(k\tau)u_0\|_{L^{r_0}}^{q_0}\)^{1/q_0}
    &  \le  \uC \|u_0\|_{L^2} 
   + C\left\| \frac{N_0(\tau)-\Id}{\tau}Z_\tau
    (n\tau)u_0\right\|_{\ell^{q_0'}(0\le n\tau\le
      N_*\tau;L^{r_0'})}\\
&\quad + C \left\| \frac{e^{-i\tau\vp}-\Id}{\tau}Z_\tau
    (n\tau)u_0\right\|_{\ell^1(0\le n\tau\le
      N_*\tau;L^2)},
\end{align*}
for some $C=C(d,\si,T)$. 
Using \eqref{eq:NLpuissance}, we infer
\begin{align*}
   \(\tau\sum_{k=0}^{N_*+1}\|Z_\tau(k\tau)u_0\|_{L^{r_0}}^{q_0}\)^{1/q_0}
      & \le  \uC \|u_0\|_{L^2} 
  + C \|Z_\tau (n\tau)u_0\|_{\ell^{(2\si+1)q_0'}(0\le n\tau\le
  N_*\tau;L^{(2\si+1)r_0'})}^{2\si+1}\\
&\quad + C\|Z_\tau
    (n\tau)u_0\|_{\ell^1(0\le n\tau\le      N_*\tau;L^2) }.
\end{align*}
We note that $(2\si+1)r_0'=r_0$, and, since
\begin{equation*}
  (2\si+1)\(\frac{1}{(2\si+1)q_0'}-\frac{1}{q_0}\)=1-\frac{2\si+2}{q_0}=
  1-\frac{d\si}{2},
\end{equation*}
H\"older inequality yields
\begin{equation*}
   \|Z_\tau (n\tau)u_0\|_{\ell^{(2\si+1)q_0'}(I;L^{(2\si+1)r_0'})}^{2\si+1}\le
|I|^{1-\frac{d\si}{2}} \|Z_\tau
(n\tau)u_0\|_{\ell^{q_0}(I;L^{r_0})}^{2\si+1}.
\end{equation*}
We infer
\begin{align*}
 \|Z_\tau(n\tau)u_0\|_{\ell^{q_0}(0\le n\tau\le (N_*+1)\tau;L^{r_0})}
    &   \le  \uC \|u_0\|_{L^2} + C(T)\|Z_\tau
(n\tau)u_0\|_{\ell^{q_0}(0\le n\tau\le N_*\tau;L^{r_0})}^{2\si+1}\\
&\quad +C(T)\|u_0\|_{L^2}.
\end{align*}
The definition of $N_*$ and Theorem~\ref{theo:stab-cv} imply, for
$\tau\le 1$,
\begin{align*}
 \|Z_\tau(n\tau)u_0\|_{\ell^{q_0}(0\le n\tau\le (N_*+1)\tau;L^{r_0})}
    &   \le  \uC \|u_0\|_{L^2} + C(T)\|u(n\tau)\|_{\ell^{q_0}(0\le
      n\tau\le N_*\tau;L^{r_0})}^{2\si+1}\\ 
&\quad +C(T)\|u_0\|_{L^2}+ C_2 \tau^{1/2}.
\end{align*}
We can then set 
\begin{equation*}
  C_1 = 2 C(T) + 2C(T) \|u(n\tau)\|_{\ell^{q_0}(0\le
      n\tau\le T;L^{r_0})}^{2\si},
\end{equation*}
so 
\begin{align*}
 \|Z_\tau(n\tau)u_0\|_{\ell^{q_0}(0\le n\tau\le (N_*+1)\tau;L^{r_0})}
    &   \le  \uC \|u_0\|_{L^2} +\frac{C_1}{2}\|u_0\|_{L^2} \\
&\quad+
      \frac{C_1}{2}\|u(n\tau)\|_{\ell^{q_0}(0\le 
      n\tau\le N_*\tau;L^{r_0})}+ C_2 \tau^{1/2}.
\end{align*}
Taking $\tau>0$ sufficiently small then contradicts 
the maximality of $N_*$, hence the lemma.
\end{proof}

\subsection{Case $\si\ge 2/d$}

 When $V=0$, the case $u\in X_T^1$
is considered in \cite{ChoiKoh2021}: a central role is played by
Proposition~5.1 there, where the authors prove a local stability
result at the $H^1$ level, based on continuity arguments and the
density of $H^2(\R^d)$ in $H^1(\R^d)$. Typically, (5.5) in
\cite{ChoiKoh2021} is exactly the convergence stated in the last
point of Theorem~\ref{theo:stab-cv} here. In the present context,  we
shall therefore prove that there exists $M$ 
such that
\begin{equation}\label{eq:goal-stab-H}
  \|H^{1/2}u^n\|_{\ell^\infty([0,T];L^2)} =
  \|H^{1/2}Z_\tau(n\tau)\|_{\ell^\infty([0,T];L^2)}  \le M. 
\end{equation}
We first consider the case $u\in X_T^2$: 
\begin{lemma}[Local stability, $L^2$-(super)critical
  case]\label{lem:local-stab-sup} 
  Let $2/d\le\si<2/(d-2)_+$, with $\si\ge 1/2$, $u_0\in \cH^2$, and $T>0$ such
  that $u\in X_T^2$. There exists $\tau_0>0$ such that for all
  $\tau\in (0,\tau_0]$, $Z_\tau (n\tau)u_0\in
  \ell^\infty([0,T];\cH^1)$, and there exists $C>0$ such that
  \begin{equation*}
    \|H^{1/2}Z_\tau (n\tau)u_0\|_{\ell^\infty([0,T];L^2)}\le C
    \|u\|_{X_T^1},.
  \end{equation*}
 \end{lemma}
\begin{proof}
 The spirit of the proof is the same as for Lemma~\ref{lem:stab-sub}. Set
\begin{align*}
  \Lambda =\Big\{N\in \N,\quad
  &   \|H^{1/2}Z_\tau(n\tau)u_0\|_{\ell^{q_0}(0,N\tau;L^{r_0})} 
   +
    \|H^{1/2}Z_\tau(n\tau)u_0\|_{\ell^\infty(0,N\tau;L^2)}\\
&\le 
     \uK \|H^{1/2}u_0\|_{L^2}+C_1 \|u\|_{X_T^1}
  \Big\},
\end{align*}
where $\uK=\uK(d,\si,T)$ is defined by
\begin{equation*}
  \uK= \sup_{\tau\in (0,1)}\sup_{\phi\in \cH^1}\frac{
    \|H^{1/2}S_{1/\tau}\phi \|_{\ell^{q_0}([0,T];L^{r_0})} +
    \|H^{1/2}S_{1/\tau}
    \phi\|_{\ell^\infty([0,T];L^2)}}{\|H^{1/2}\phi\|_{L^2}},
\end{equation*}  
and $C_1$ will be given in the course of the proof. 
This set is not empty, as 
\begin{equation*}
    H^{1/2}Z_\tau (0)u_0=
    H^{1/2}S_{1/\tau}u_0,
\end{equation*}
and so $0\in \Lambda$. If $\Lambda$ is infinite, then
\eqref{eq:stab-num} holds, in view of 
the above presentation. If $\Lambda$
is bounded, let $N_*=N_*(\tau)$ be its largest element and we suppose that
 $N_*+1< T/\tau$.

Strichartz estimates yield, for
$(q,r)\in\{(q_0,r_0),(\infty,2)\}$, 
\begin{align*}
  \|H^{1/2}Z_\tau(n\tau)u_0 & \|_{\ell^{q}([0,(N_*+1)\tau];L^{r})}
  \le \uK  \|H^{1/2}u_0\|_{L^2} \\
  &+ C\left\|H^{1/2}e^{-i\tau\vp}\frac{N_0(\tau)-\Id}{\tau}
    Z_\tau(n\tau)u_0\right\|_{\ell^{q_0'}([0,N_*\tau];L^{r_0'})}\\
&+ C\left\|H^{1/2}\frac{e^{-i\tau\vp}-\Id}{\tau}
    Z_\tau(n\tau)u_0\right\|_{\ell^1([0,N_*\tau];L^2)}. 
\end{align*}
Invoking Lemmas~\ref{lem:sobolev-H} and \ref{lem:propNL}, 
\begin{align*}
  &\left\|H^{1/2}\frac{N_0(\tau)-\Id}{\tau}
    Z_\tau(n\tau)u_0\right\|_{\ell^{q_0'}([0,N_*\tau];L^{r_0'})}\\
  &\qquad\lesssim \sum_{J\in \{\nabla,\sqrt V,\Id\}}\left\||Z_\tau(n\tau)u_0|^{2\si} J
    Z_\tau(n\tau)u_0\right\|_{\ell^{q_0'}([0,N_*\tau];L^{r_0'})}\\
  &\qquad\lesssim \|Z_\tau(n\tau)
    u_0\|_{\ell^{\theta}([0,N_*\tau];L^{r_0})}^{2\si} 
    \sum_{J\in \{\nabla,\sqrt V,\Id\}}\|J
    Z_\tau(n\tau)u_0\|_{\ell^{q_0}([0,N_*\tau];L^{r_0})}\\
    &\qquad\lesssim \|Z_\tau(n\tau)
    u_0\|_{\ell^{\theta}([0,N_*\tau];L^{r_0})}^{2\si} 
   \|H^{1/2}
    Z_\tau(n\tau)u_0\|_{\ell^{q_0}([0,N_*\tau];L^{r_0})}
\end{align*}
On the other hand,
\begin{align*}
 \|Z_\tau(n\tau)
    u_0\|_{\ell^{\theta}([0,N_*\tau];L^{r_0})}&\lesssim  \|Z_\tau(n\tau)
 u_0\|_{\ell^{\theta}([0,N_*\tau];H^1)}\\
  &\le C(T) \|Z_\tau(n\tau)
 u_0\|_{\ell^\infty([0,N_*\tau];H^1)}\\
&\le C(T)\|H^{1/2}Z_\tau(n\tau)
  u_0\|_{\ell^\infty([0,N_*\tau];L^2)}.
\end{align*}
The factor $e^{-i\tau\vp}$ does not change the nature of this
estimate, and we find, setting $T_*=(N_*+1)\tau$,  for some constants
depending on $T$, 
\begin{align*}
  \|&H^{1/2}Z_\tau(n\tau)u_0\|_{\ell^{q}([0,T_*];L^{r})} \le
   \uK   \|H^{1/2}u_0\|_{L^2}\\
  &\qquad+
 C  \|H^{1/2}Z_\tau(n\tau)u_0\|^{2\si}_{\ell^\infty(
   [0,N_*\tau];L^2)}
    \|H^{1/2}Z_\tau(n\tau)u_0\|_{\ell^{q_0}([0,N_*\tau];L^{r_0})}  \\
&\qquad+
 C \|H^{1/2}Z_\tau(n\tau)u_0\|_{\ell^\infty([0,N_*\tau];L^2)}  .
\end{align*}
Summing the cases
$(q,r)=(q_0,r_0)$ and $(q,r)=(\infty,2)$, we infer
\begin{align*}
  \sum_{(q,r)\in \{(q_0,r_0),(\infty,2)\}}\|&H^{1/2}Z_\tau(n\tau)u_0\|_{\ell^{q}([0,T_*];L^{r})} \le
   2\uK   \|H^{1/2}u_0\|_{L^2}\\
  &+
 C  \|H^{1/2}Z_\tau(n\tau)u_0\|^{2\si}_{\ell^\infty(
   [0,N_*\tau];L^2)}
    \|H^{1/2}Z_\tau(n\tau)u_0\|_{\ell^{q_0}([0,N_*\tau];L^{r_0})}  \\
&+
 C \|H^{1/2}Z_\tau(n\tau)u_0\|_{\ell^\infty([0,N_*\tau];L^2)}  .
\end{align*}
We then set
\begin{equation*}
  C_1 = 2C +2C \|H^{1/2}u(n\tau),\|^{2\si}_{\ell^\infty(
   [0,T];L^2)},
\end{equation*}
so the definition of $N_*$ and Theorem~\ref{theo:stab-cv} yield, for
all $\tau\le 1$,
\begin{align*}
  \sum_{(q,r)\in \{(q_0,r_0),(\infty,2)\}}\|&H^{1/2}Z_\tau(n\tau)u_0\|_{\ell^{q}([0,T_*];L^{r})} \le
   2\uK   \|H^{1/2}u_0\|_{L^2}\\
  &+
\frac{C_1}{2}   \|H^{1/2}u(n\tau)\|_{\ell^{q_0}([0,N_*\tau];L^{r_0})}  \\
&+
 \frac{C_1}{2} \|H^{1/2}u(n\tau)\|_{\ell^\infty([0,N_*\tau];L^2)}  +C\tau^{1/2}.
\end{align*}
For $\tau>0$ sufficiently small, the maximality of
$N_*$ leads to a contradiction, hence the lemma.
\end{proof}

We have assumed so far $u\in X_T^2$. In Theorem~\ref{theo:main}, we
assume only $u\in X_T^1$, and the nonlinearity is supposed to be $C^2$
(or almost) when $\si\ge 2/d$: like in \cite{ChoiKoh2021}, passing
from the $X_T^2$ case to the $X_T^1$ case relies on a density
argument:
\begin{proposition}\label{prop:density1}
   Let Assumption~\ref{hyp:V} be verified and $\frac{1}{2}\le
   \si<\frac{2}{(d-2)_+}$. For any $M\ge 1$, there exists $T>0$ and
   $C=C(M,d,\si)$ such that if $u_0,v_0\in \cH^1$ with
   $\|u_0\|_{\cH^1},\|v_0\|_{\cH^1}\le M$, then the solutions $u$ and
   $v$ to \eqref{eq:NLSP} with respective initial data $u_0$ and $v_0$
   are such that $u,v\in X_T^1$, and
   \begin{equation*}
     \|u-v\|_{X_T^1}\le C\|u_0-v_0\|_{\cH^1}. 
   \end{equation*}
\end{proposition}
We refer to \cite[Theorem~2.3]{CaSu24} for a proof which is readily
  adapted to the present framework. We also invoke the following
  result, whose proof is 
  essentially the same as for \cite[Proposition~5.1, (5.4)]{ChoiKoh2021} or
  \cite[Proposition~6.2, (6.4)]{CaSu24}:
  \begin{proposition}\label{prop:density2}
      Let Assumption~\ref{hyp:V} be verified and $\frac{1}{2}\le
   \si<\frac{2}{(d-2)_+}$. For any $M\ge 1$, there exists $T>0$ and
   $C=C(M,d,\si)$ such that if $u_0,v_0\in \cH^1$ with
   $\|u_0\|_{\cH^1},\|v_0\|_{\cH^1}\le M$, then for all $\tau\in
   (0,1]$, and all admissible pair $(q,r)$,
   \begin{equation*}
     \left\|H^{1/2}\(Z_\tau(n\tau)  u_0-Z_\tau(n\tau)
       v_0\)\right\|_{\ell^q([0,T];L^r)}\le C\|u_0-v_0\|_{\cH^1}.  
   \end{equation*}
  \end{proposition}
In view of Theorem~\ref{theo:stab-cv},
Propositions~\ref{prop:density1} and \ref{prop:density2}, we have
(like in \cite[Proposition~5.1, (5.6)]{ChoiKoh2021}), by approaching
$u_0\in \cH_1$ by $v_0\in \cH^2$ with $\|u_0-v_0\|_{\cH^1}\le \eps$
and eventually letting $\eps$ go to zero:
\begin{equation*}
  \lim_{\tau \to 0} \|Z_\tau(n\tau)u_0-u(n\tau)\|_{\ell^\infty([0,T];\cH^1)}=0.
\end{equation*}
In particular, for $\tau_0>0$ sufficiently small and $\tau\in
(0,\tau_0]$, we know that
\begin{equation*}
 \|Z_\tau(n\tau)u_0\|_{\ell^\infty([0,T];\cH^1)}\le M 
\end{equation*}
 for some $M$ independent of $\tau\in (0,\tau_0]$, which is (slightly
 more than) what we wanted to prove.

\subsection*{Acknowledgments}
The author wishes to thank Valeria Banica, Georg Maierhofer and
Katharina Schratz for interactions which motivated this paper, as well
as 
Julien Sabin and San V\~u Ng\d oc for instructive discussions on the
material of Section~\ref{sec:tools}.

\bibliographystyle{abbrv}
\bibliography{biblio}

\end{document}